\documentclass[10pt,reqno]{amsart}
\usepackage{amssymb,mathrsfs}
\usepackage{graphicx,cite,times}
\usepackage[
colorlinks,
citecolor=blue,
linkcolor=blue,
backref=page
allcolors=black
]{hyperref}
\usepackage{array,color}

\usepackage[hyphenbreaks]{breakurl}
\newtheorem{theorem}{Theorem}[section]
\newtheorem{corollary}[theorem]{Corollary}
\newtheorem{lemma}[theorem]{Lemma}

\newtheorem{remark}[theorem]{Remark}

\newtheorem{definition}[theorem]{Definition}

\numberwithin{equation}{section}

\allowdisplaybreaks

\def\N{\mathbb N}

\def\T{\mathbb T}

\def\Z{\mathbb Z}

\begin{document}
\large
\title[Exponential convergence of weighted Birkhoff average]{Exponential convergence of weighted Birkhoff average}

\author{Zhicheng Tong}
\address{\scriptsize (Z. C. Tong)~College of Mathematics, Jilin University, Changchun 130012, P.R. China.}
\email{tongzc20@mails.jlu.edu.cn}

\author{Yong Li}
\address{\scriptsize  (Y. Li)~Institute of Mathematics, Jilin University, Changchun 130012, P.R. China. School of Mathematics and Statistics, Center for Mathematics and Interdisciplinary Sciences, Northeast Normal University, Changchun, Jilin 130024, P.R.China.}
\email{liyong@jlu.edu.cn}
\thanks{The second author (Y. Li) was supported in part by National Basic Research Program of China Grant (2013CB834100) and NSFC Grant (12071175, 11171132,  11571065), Project of Science and Technology Development of Jilin Province (2017C028-1, 20190201302JC), and Natural Science Foundation of Jilin Province (20200201253JC)}



\begin{abstract}
In this paper, we consider the polynomial and exponential convergence rate of weighted Birkhoff averages of irrational rotations on tori. It is shown that these can be achieved for finite and infinite dimensional tori which correspond to the quasiperiodic and almost periodic dynamical systems respectively,  under certain balance between the nonresonant condition and the decay rate of the Fourier coefficients. Diophantine rotations with finite and infinite dimensions are provided as examples. For the first time, we prove the universality of exponential convergence and arbitrary polynomial convergence in the quasiperiodic case and almost periodic case under analyticity respectively.
\end{abstract}

\maketitle
\tableofcontents

\section{Introduction}
The classical Birkhoff ergodic theorem asserts that for ergodic dynamical systems, the time average of a function $ f $ evaluated along a trajectory of length $ N $ converges to the space average, i.e., the integral of $ f $ over the space. Namely, assume $ T: X \to X $ is a map on a topological space $ X $ with a  probability measure $ \mu $ for which $ T $ is invariant. Then for a fixed point $ x \in X $ and a function $ f $ on $ X $, we define the long time average of $ f $ as
\begin{equation}\label{BBB}
	{\mathrm{B}_N}\left( f \right)\left( x \right): = \frac{1}{N}\sum\limits_{n = 0}^{N - 1} {f\left( {{T^n}\left( x \right)} \right)} ,
\end{equation}
which we call the Birkhoff average of $ f $. Actually, it has a long history to study the convergence of \eqref{BBB}, see  survey articles   \cite{Mackey,ccm}.  The von Neumann ergodic theorem shows that \eqref{BBB} converges to the integral $ \int_X {fd\mu }  $ in the $ L^2 $ norm, if $ f \in L^2(X,\mu) $, $ \mu $ is a  probability measure on $ X $, $ T $ preserves $ \mu $ and is ergodic, see Theorem 4.5.2 given in \cite{MR1963683}.  The Birkhoff  ergodic theorem weakens the restriction of the former on $ f $, only $ f \in L^1(X,\mu) $ is required, then \eqref{BBB} converges to  $ \int_X {fd\mu }  $, $ \mu\text{-a.e.} $ on $ X $. These theorems are of great importance both in mathematics and statistical mechanics. However, the convergence rate of the Birkhoff average  may be very slow. It can be proved that for any non-constant $ f $, there exists a constant independent of $ N $, such that
\begin{equation}\label{renyi2}
	\left| {{\mathrm{B}_N}\left( f \right)\left( x \right) - \int_X {fd\mu } } \right| \geqslant \frac{C}{N}
\end{equation}
holds for infinitely many $ N $, see \cite{MR3755876}. In fact,  many mathematicians who have worked on Birkhoff  ergodic theorem know that it is not possible to prove any general positive result about the speed of convergence in \eqref{BBB}, and later it has been shown in \cite{MR0510630} that for any null-sequence $ \left\{ {{\omega _n}} \right\}_{n = 1}^\infty  $ of positive reals, there exists a continuous function $ f $ such that
\begin{equation}\label{renyi}
	\mathop {\lim \sup }\limits_{N \in {\mathbb{N}^ + }} \omega _N^{ - 1}\left| {{\mathrm{B}_N}\left( f \right)\left( x \right) - \int_X {fd\mu } } \right| =  + \infty ,\;\;a.e.\;
\end{equation}
And the analogous result holds also for norm-convergence.

Obviously, the slow rate of  the convergence of the  Birkhoff average \eqref{BBB} makes numerical computations in real problems extremely difficult, although the convergence is guaranteed in theory. Aiming to get high precision numerical results, some computations may even take billions of years to complete, see Subsection 1.9 in \cite{MR3755876} and \cite{MR3718733}. This forces ones  to find a faster convergence method, from which some weighted Birkhoff averages have been derived.

Recently, a weighted method of non-uniform distribution was proposed in \cite{MR3755876} to study ergodicity in  quasiperiodic dynamical systems, which surprisingly confirms that when $ f:\mathbb{T}^d \to E $ is sufficiently smooth, provided $ d \in \mathbb{N}^+ $ and $ \dim E <  + \infty  $, and the rotation vector on $ \T^d $ satisfies the  Diophantine condition, then the weighted Birkhoff average could converge at an  arbitrarily polynomial rate which they called super-convergence. \textit{This is indeed a breakthrough.} See \cite{MR3718733} for numerical simulation of some physical models. At this point, it is therefore natural that ones should consider the following questions step by step:

\begin{itemize}
\item[(\textbf{Q1})] \textit{How about the convergence type in the almost periodic case?}

\item[(\textbf{Q2})] \textit{Could faster convergence than arbitrary polynomial's type be achieved, such as exponential's type?}

\item[(\textbf{Q3})] \textit{Can we show certain  universality of arbitrary polynomial convergence and exponential convergence via analyticity?}
\end{itemize}

These questions are quite nontrivial. On the one hand, \textit{the almost periodic case is fundamentally different from the quasiperiodic case} in that the rotations of the former  are infinite-dimensional vectors, while the latter only deals with  finite-dimensional rotations, see \cite{Herman} and \cite{Moser} for relevant work on these two aspects. Their topological properties are completely different, such as the infinite-dimensional torus has no compactness.  Additionally, the data processing is even more different, that is, the almost periodic case may lead to \textit{Curse of Dimensionality}. On the other hand, (\textbf{Q1}), (\textbf{Q2}) and (\textbf{Q3}) are crucial both theoretically and applicability, and they also explain Laskar's simulation results \cite{MR1720890} (Remark 2 in Appendix, p.146) on quasiperiodic flows, that is, convergence faster than arbitrary polynomial's type. In this paper, we  make further developments following  \cite{MR3755876} and answer these questions.

This paper is organized as follows. In Section \ref{sec2}, we show that the weighted Birkhoff average converges at an arbitrary polynomial rate for the quasiperiodic case ($ \mathbb{T}^d $) and the almost periodic case ($ \mathbb{T}^\infty $), as long as the nonresonance of the irrational rotation vector and the Fourier coefficients of $ f $ satisfy certain conditions, i.e., (\textbf{H1}) and (\textbf{H2}), respectively. Roughly speaking, the rotating vector might satisfy weaker nonresonance than the usual Diophantine one, particularly involving with infinite-dimensional cases. In Section \ref{sec3}, we present our main results in this paper, which further show that the weighted Birkhoff average can indeed  converge at an exponential rate, by requiring stronger conditions than that before. Diophantine rotations  are constructed as examples at this point, including finite and infinite dimensional cases. As a corollary to the above results, we show that under the assumption of analyticity, exponential convergence and arbitrary polynomial convergence are universal in the case of quasiperiodic and almost periodic,  respectively. It is worth mentioning that some assumptions can be removed in the  cases without small divisors, in dealing with exponential convergence. It seems inevitable, however, that the difficulty of analyzing exponential convergence and circumventing the limitation of dimensionality lead to technical complications.

\section{Convergence of arbitrary polynomial  rate type} \label{sec2}
\subsection{Finite-dimensional case $ \mathbb{T}^d $}
Das and Yorke \cite{MR3755876} proved arbitrary polynomial convergence in the quasiperiodic case  via Diophantine rotations and $ C^\infty $ regularity. Following their idea, we extend the results to the general nonresonant conditions and regularity of functions, as well as the continuous case. To introduce the results we first give some notions, which are basic to our discussion.
\begin{definition}
	A function $ \Delta :\left[ {1, + \infty } \right) \to \left[ {1, + \infty } \right) $  is said to be an approximation function, if it is continuous, strictly monotonic increasing, and satisfies  $  \Delta(+\infty) =+\infty$.
\end{definition}

\begin{definition}[Finite-dimensional nonresonant condition] \label{Fi}
	An irrational vector $ \rho \in \mathbb{T}^d $ is said to be nonresonant if there exist  $ \alpha >0 $ and an approximation function $ \Delta $ such that
	\begin{itemize}

		\item[($ a $)] The discrete case
		\begin{equation}\label{nonresonant}
			\left| {k \cdot \rho  - n} \right| \geqslant \frac{\alpha }{{\Delta \left( {||k||} \right)}},\;\;\forall 0 \ne k \in {\mathbb{Z}^d},\;\;\forall n \in \mathbb{Z};
		\end{equation}
		\item[($ b $)] The continuous case
		\begin{equation}\label{nonresonant2}
			\left| {k \cdot \rho } \right| \geqslant \frac{\alpha }{{\Delta \left( {||k||} \right)}},\;\;\forall 0 \ne k \in {\mathbb{Z}^d},
		\end{equation}
	\end{itemize}
	where $\left\|k\right\|=|k_1|+\cdots+|k_d|$.
\end{definition}
\begin{remark}
	Obviously \eqref{nonresonant} implies \eqref{nonresonant2}, which means that in the continuous case one has a lower restriction on rotation vectors and thus we state them separately.
\end{remark}
\begin{remark}
	We say that $ \rho\in \mathbb{T}^d $ satisfies the Finite-dimensional Diophantine condition, if
	\begin{equation}\label{Finite-dimensional Diophantine condition}
		\Delta(x)=x^\tau,\;\;\tau > d -1.
	\end{equation}
\end{remark}

\begin{definition}
	Assume $ \left( {\mathcal{B},|| \cdot ||{_\mathcal{B}}} \right) $ is a 
	Banach function space (could be infinite-dimensional), and  $ f:{\mathbb{T}^d} \to \mathcal{B} $ with
	\begin{equation}\label{ff}
		f = \sum\limits_{ k \in {\mathbb{Z}^d}} {{{\hat f}_k}{e^{2\pi ik \cdot \theta }}} ,\;\;{{\hat f}_k} = \int_{{\mathbb{T}^d}} {f( {\hat \theta } ){e^{ - 2\pi ik \cdot \hat \theta }}d\hat \theta },
	\end{equation}
	where the first "$ = $" represents equality in the sense of the norm $ || \cdot ||{_\mathcal{B}} $.
	Now we define the following  space 
	\begin{equation}\label{Bdelta}
		{\mathcal{B}_{\tilde \Delta }}: = \left\{ {f:{\mathbb{T}^d} \to \mathcal{B} :\text{$ f  $ satisfies \eqref{ff} and $ \mathop {\sup }\limits_{0 \ne k \in {\mathbb{Z}^d}} \tilde \Delta \left( {||k||} \right)\| {{{\hat f}_k}} \|_{_\mathcal{B}} <  + \infty $} } \right\}
	\end{equation}
	for a given approximation function $ {\tilde \Delta } $. 
\end{definition}


For a given map $ {T_\rho }:{\mathbb{T}^d} \to {\mathbb{T}^d} $ with $ {T_\rho }\left( \theta  \right) = \theta  + \rho \bmod 1 $  in each coordinate ($\rho$ is an  irrational nonresonant vector) and a function $ f \in \mathcal{B}_{\tilde \Delta }$, define the weighted Birkhoff average as
\begin{equation}\label{birkhoff}
	{\mathrm{WB}_N}\left( f \right)\left( \theta  \right): = \frac{1}{{{A_N}}}\sum\limits_{n = 0}^{N - 1} {w\left( {\frac{n}{N}} \right)f\left( {T_\rho ^n}(\theta) \right)} ,\;\;{A_N} = \sum\limits_{n = 0}^{N - 1} {w\left( {\frac{n}{N}} \right)} ,\;\;\theta \in \mathbb{T}^d,
\end{equation}
where $ w$ is a $ \mathcal{C}_0^m\left( {\left[ {0,1} \right]} \right) $ weighting function with $ 2 \leqslant m \leqslant \infty  $, that is, $ w \in {C^\infty }\left( {\left[ {0,1} \right]} \right) $, $ {w^{\left( k \right)}}\left( 0 \right) = {w^{\left( k \right)}}\left( 1 \right) = 0 $ for all $ 0 \leqslant k \leqslant m $, $ w(x)>0 $ for $ x\in (0,1) $, and $ \int_0^1 {w\left( s \right)ds}  =1 $.

We make the following assumption:

(\textbf{H1})\ The approximation functions given in \eqref{nonresonant}, \eqref{nonresonant2} and \eqref{Bdelta} satisfy the integrability condition:
\[\int_1^{ + \infty } {\frac{{{r^{d - 1}}{\Delta ^m}\left( r \right)}}{{\tilde \Delta \left( r \right)}}dr}  <  + \infty .\]

\begin{theorem}\label{T1}
	Give $ f \in \mathcal{B}_{\tilde \Delta }  $, and $ \rho $ satisfies the Finite-dimensional nonresonant condition in Definition \ref{Fi}. Assume (\textbf{H1}). Then there hold
	\begin{equation}\label{a.e.}
		{\left\| {{\mathrm{WB}_N}\left( f \right)\left( \theta  \right) - \int_{{\mathbb{T}^d}} {f( {\hat \theta } )d\hat \theta } } \right\|_\mathcal{B}} \leqslant \frac{{{C_1 }}}{{{N^m}}},\;\;N \geqslant 1,
	\end{equation}
	and
	\begin{equation}\label{T1-2}
		{\left\| {\frac{1}{T}\int_0^T {w \left( t/T \right)f\left( {\rho t + \theta } \right)dt}  - \int_{{\mathbb{T}^d}} {f( {\hat \theta } )d\hat \theta } } \right\|_\mathcal{B}} \leqslant \frac{{{C_1 }}}{{{T^m}}},\;\;T \geqslant 1,
	\end{equation}
	where the positive constant $ C_1>0 $ only depends on $ f,\Delta ,\tilde \Delta ,w,\alpha ,m,d $.
\end{theorem}

Let us make some comments.

\begin{itemize}

	\item[(1)] As long as the weighting function $ w $ is sufficiently smooth, the convergence rate of weighted Birkhoff  average can reach arbitrarily polynomial convergence. For example, we could take
	\begin{equation}\label{zhishupingjun}
		\tilde{w}\left( x \right):= \left\{ \begin{array}{ll}
			{\left( {\int_0^1 {\exp \left( { - {s^{ - p}}{{\left( {1 - s} \right)}^{ - q}}} \right)ds} } \right)^{ - 1}}\exp \left( { - {x^{ - p}}{{\left( {1 - x} \right)}^{ - q}}} \right),&x \in \left( {0,1} \right),  \hfill \\
			0,&x = 0,1, \hfill \\
		\end{array}  \right.
	\end{equation}
	for any $ p,q>0 $. One can easily verify that $ \tilde{w} $ is $ \mathcal{C}_0^\infty \left( {\left[ {0,1} \right]} \right) $. It turns out that the relationship between the computational convergence rate and $ p,q $ is not particularly clear, see \cite{MR3718733}.
	
	\item[(2)] Spatial structure \eqref{Bdelta} of $ \mathcal{B}_{\tilde \Delta } $ and condition (\textbf{H1}) together guarantee the uniform convergence of weighted Birkhoff average \eqref{a.e.}. However, after removing \eqref{ff} in \eqref{Bdelta}, then \eqref{a.e.} may only hold a.e. on $ \mathbb{T}^d $, because for $ f $ in general, the  Fourier series of $ f $ does not necessarily pointwise converge to $ f $, so we cannot obtain uniform convergence of the weighted Birkhoff average with respect to all $ \theta \in \mathbb{T}^d $. If $ f \in {\mathbb{R}^l} $ with $ l \in \mathbb{N}^+ $ is continuous and satisfies $ f\left( 0 \right) = f\left( 1 \right) $ in $ \mathbb{R}^l $ (here $ 0 = \left( {0, \ldots ,0} \right),1 = \left( {1, \ldots ,1} \right) $ in $ \mathbb{T}^d $), then \eqref{a.e.} could indeed converge uniformly. At this point, if we further assume that:
	
	(i) $ f $ is $ C^M $ smooth, then $ | {{\hat f_k}} | \leqslant {C_{f,M}}||k||{^{ - M}} $ can be obtained by integration by parts for all $ 0 \ne k \in {\mathbb{Z}^d} $, that is, $ \tilde \Delta \left( x \right): = {x^M} $ and $ f \in {\mathcal{B}_{\tilde \Delta }} $. Obviously, if the  rotation vector $ \rho $ is Diophantine, i.e., $ \Delta \left( x \right): = {x^\tau } $ with $ \tau>d-1 $, then (\textbf{H1}) can be satisfied as long as $ M > d + m\tau  $, because
	\[\int_1^{ + \infty } {\frac{{{r^{d - 1}}{\Delta ^m}\left( r \right)}}{{\tilde \Delta \left( r \right)}}dr}  = \int_1^{ + \infty } {\frac{1}{{{r^{M + 1 - d - m\tau }}}}dr}  <  + \infty .\]
	This is the case given in \cite{MR3755876}.
	
	(ii) $ f $ is Gevrey smooth in some neighbourhood of $ \mathbb{T}^d $ in $ \mathbb{C}^d $, i.e., there exist $ c_f,\mu  > 0 $ and $ \nu  \in \left( {0,1} \right] $ such that $ | {{{\hat f}_k}} | \leqslant c_f{e^{ - \mu ||k||{^\nu }}} $ for all $ 0 \ne k \in {\mathbb{Z}^d} $. In particular, $ f $ is analytic when $ \nu =1 $. This leads to $ \tilde \Delta \left( x \right): = {e^{\mu {x^\nu }}} $. Therefore, if the  rotation vector $ \rho $ satisfies the nonresonant conditions \eqref{nonresonant}, \eqref{nonresonant2} with $ \Delta \left( x \right) = {e^{\tilde \mu {x^\nu }}},0 < \tilde \mu  < {m^{ - 1}}\mu  $ (weaker than the Diophantine type), then  one can verify (\textbf{H1}) as:
	\begin{align*}
	\int_1^{ + \infty } {\frac{{{r^{d - 1}}{\Delta ^m}\left( r \right)}}{{\tilde \Delta \left( r \right)}}dr}  &\leqslant \int_1^{ + \infty } {\frac{{{r^{d - 1}}{e^{m\tilde \mu {r^\nu }}}}}{{{e^{\mu {r^\nu }}}}}dr} \\
	& \leqslant {c_{d,\mu ,m,\tilde \mu ,\nu }}\int_1^{ + \infty } {\frac{1}{{{e^{\left( {\mu  - m\tilde \mu } \right){r^\nu }/2}}}}dr}  \\
	&<  + \infty .
	\end{align*}

	For two cases given above, the uniform convergence of weighted Birkhoff average can be obtained by applying Theorem \ref{T1}, and the convergence rate is polynomial.
	
	\item[(3)] It should be pointed out that we generalize the Diophantine condition for the irrational rotation vector $ \rho $  since the rapid convergence of the Fourier coefficients of $ f $ could overcome the nonresonance of $ \rho $, as shown in (ii). For example, for a given approximation function $  \Delta \left( x \right) $ in \eqref{nonresonant} and \eqref{nonresonant2}, we can require that  the Fourier coefficients of $ f $ to converge rapidly to
	\[\tilde \Delta \left( x \right) \sim {\Delta ^m}\left( x \right){x^d}(\log x)(\log \log x) \cdots {(\underbrace {\log  \cdots \log }_{\ell  \in {\mathbb{N}^ + }}x)^{1 + \zeta }},\;\;x \to  + \infty \]
	with $ \ell  \in {\mathbb{N}^ + } $ and $ \zeta>0 $ in \eqref{Bdelta}. Then (\textbf{H1}) holds because
	\begin{align*}
	\int_1^{ + \infty } {\frac{{{r^{d - 1}}{\Delta ^m}\left( r \right)}}{{\tilde \Delta \left( r \right)}}dr}  &= \mathcal{O}\left( {\int_M^{ + \infty } {\frac{1}{{r(\log r)(\log \log r) \cdots {{(\underbrace {\log  \cdots \log }_{\ell  \in {\mathbb{N}^ + }}r)}^{1 + \zeta }}}}dr} } \right) \\
	&<  + \infty .
	\end{align*}
	
	\item[(4)] For the continuous case \eqref{T1-2}, we only require that the nonresonant condition \eqref{nonresonant}  holds for $ n=0 $, i.e., \eqref{nonresonant2}, since we just have to integrate by parts directly with respect to $ \int_0^T {w \left( {t/T} \right){e^{2\pi itk \cdot \rho }}dt}  $ in proof.

	\item[(5)] The dependence of the universal  constant $ C_1>0 $ on the spatial dimension $ d $ is actually caused by the integrability assumption (\textbf{H1}) which is somewhat easier to verify. As to the subsequent infinite dimensional cases (e.g., Theorem \ref{T3}), one  has to only require the boundedness for series (e.g., (\textbf{H2})) to eliminate the influence of dimension, which is not essential. Additionally, one observes that $ C_1 $ might tend to infinite (e.g., it can be verified that $ \mathop {\lim }\limits_{m \to  + \infty } {\left\| {{{\bar w}^{\left( m \right)}}} \right\|_{{L^1}\left( {0,1} \right)}} =  + \infty  $ with the weighting function $ \bar{w} $ in \eqref{barwde}), so if we want to achieve the exponential convergence, some special techniques are needed, as shown in Section \ref{sec3} and Subsection \ref{ProofofT3}.
\end{itemize}

\subsection{Infinite-dimensional case $ \mathbb{T}^\infty $}
However, when considering the weighted Birkhoff average \eqref{birkhoff} on the infinite-dimensional torus $ {\mathbb{T}^\infty }: = {\mathbb{T}^\mathbb{N}} $, some spatial structure has to be required. For convenience, we use the Diophantine condition for the irrational vectors $ \rho $  proposed by Bourgain and the corresponding metric, see \cite{MR2180074,MR4201442}. 

More precisely, our set of irrational vectors $ \rho $  is the infinite-dimensional cube $ {\left[ {1,2} \right]^\mathbb{N}} $ (equal to $ \mathbb{T}^\infty $), endowed  with the probability measure $ \mathbb{P} $ induced by the product measure of the infinite-dimensional cube $ {\left[ {1,2} \right]^\mathbb{N}} $. Next, for fixed $ 2 \leqslant \eta \in \mathbb{N}^+ $, we define the set of infinite integer vectors with finite support
\begin{equation}\label{Zwuqiong}
	\mathbb{Z}_ * ^\infty : = \left\{ {k \in {\mathbb{Z}^\mathbb{N}}:{{\left| k \right|}_\eta }: = \sum\limits_{j \in \mathbb{N}} {{{\left\langle j \right\rangle }^\eta }\left| {{k_j}} \right|}  <  + \infty ,\;\;\left\langle j \right\rangle : = \max \left\{ {1,\left| j \right|} \right\}} \right\}.
\end{equation}
At this point,  $ {k_j} \ne 0 $ only for finitely many indices $ j \in \mathbb{N} $. It can be seen later that the such a metric like $ {{{\left| k \right|}_\eta }} $ is necessary for the infinite-dimensional case since it determines the boundedness of the summation in proof. Besides, the infinite-dimensional analyticity also depends on the above framework, see Corollary \ref{C1}.



\begin{definition}[Infinite-dimensional nonresonant condition]\label{In}
	An irrational vector $ \rho \in \mathbb{T}^\infty $  is said to satisfy the Infinite-dimensional nonresonant condition if there exist  $ \gamma >0 $ and an approximation function $ \tt{d} $ such that
	\begin{itemize}

		\item[($ c $)] The discrete case
		\begin{equation}\label{Infinite-dimensional nonresonant condition}
			\left| {k \cdot \rho  - n} \right| > \frac{\gamma }{{{\tt{d}}( {{{\left| k \right|}_\eta }} )}},\;\;\forall 0 \ne k \in \mathbb{Z}_ * ^\infty ,\;\;\forall n \in \mathbb{Z};
		\end{equation}
		
		\item[($ d $)] The continuous case
		\begin{equation}\label{Infinite-dimensional nonresonant condition2}
			\left| {k \cdot \rho } \right| > \frac{\gamma }{{{\tt{d}}( {{{\left| k \right|}_\eta }} )}},\;\;\forall 0 \ne k \in \mathbb{Z}_ * ^\infty .
		\end{equation}
	\end{itemize}
	
\end{definition}
\begin{remark}\label{wuqiongweimance}
	In particular, if
	\begin{equation}\label{diowuqiong}
		{\tt{d}}( {{{\left| k \right|}_\eta }} ) = \prod\limits_{j \in \mathbb{N}} {\left( {1 + {{\left| {{k_j}} \right|}^\mu }{{\left\langle j \right\rangle }^\mu }} \right)} ,\;\;\forall 0 \ne k \in \mathbb{Z}_ * ^\infty
	\end{equation}
	in \eqref{Infinite-dimensional nonresonant condition} with some $ \mu >1 $, then we say that the irrational vector $ \rho$ satisfies the Infinite-dimensional Diophantine condition. Define the set
	\[{\tt{D}_{\gamma ,\mu }}: = \left\{ {\rho  \in {{\left[ {1,2} \right]}^\mathbb{N}}:\rho \;\text{satisfies the Infinite-dimensional Diophantine condition}} \right\}.\]
	Then there exists a positive constant $ C(\mu) $ such that $ \mathbb{P}\left( {{{\left[ {1,2} \right]}^\mathbb{N}}\backslash {\tt{D}_{\gamma ,\mu }}} \right) \leqslant C\left( \mu  \right)\gamma  $, as proved in \cite{MR4091501,MR2180074}.
\end{remark}

Under the above spatial structure, we introduce  Fourier expansions of functions $ f\in \mathcal{B} $ on the infinite-dimensional torus $ \mathbb{T}^\infty $ below, see \cite{MR4201442} for details:
	\begin{equation}\label{wuqiongfff}
	f = \sum\limits_{ k \in {\mathbb{Z}_*^\infty}} {{{\hat f}_k}{e^{2\pi ik \cdot \theta }}} ,\;\;{{\hat f}_k} = \int_{{\mathbb{T}^\infty}} {f( {\hat \theta } ){e^{ - 2\pi ik \cdot \hat \theta }}d\hat \theta }.
\end{equation}
Now we define the function space with rapid convergence
\begin{equation}\label{wuqiongkongjian}
	\mathcal{B}_{\tilde \Delta_\infty } : = \left\{ {f:{\mathbb{T}^\infty } \to \mathcal{B}:\text{$ f $ satisfies \eqref{wuqiongfff}, and $ \mathop {\sup }\limits_{0 \ne k \in \mathbb{Z}_ * ^\infty } {{\tilde \Delta }_\infty }\left( {{{\left| k \right|}_\eta }} \right)\| {{{\hat f}_k}} \|_{\mathcal{B}} <  + \infty $}  } \right\},
\end{equation}
for an approximation function $ {\tilde \Delta }_\infty $. In order to establish a weighted Birkhoff average theorem on $ \mathbb{T}^\infty $, we have to make an assumption like (\textbf{H1}):

(\textbf{H2})\ The approximation functions given in \eqref{Infinite-dimensional nonresonant condition}, \eqref{Infinite-dimensional nonresonant condition2} and \eqref{wuqiongkongjian} satisfy the following boundedness condition:
\[\sum\limits_{0 \ne k \in \mathbb{Z}_ * ^\infty } {\frac{{{{\tt d}^m}\left( {{{\left| k \right|}_\eta }} \right)}}{{{{\tilde \Delta }_\infty }\left( {{{\left| k \right|}_\eta }} \right)}}}  <  + \infty .\]

After the above preparation, we are in a position to give the following theorem.

\begin{theorem}\label{T3}
	Give $ f \in \mathcal{B}_{\tilde \Delta_\infty }  $, and $ \rho $ satisfies the  Infinite-dimensional nonresonant condition in Definition  \ref{In}.	Assume (\textbf{H2}). Then there hold
	\begin{equation}\label{333}
		{\left\| {{\mathrm{WB}_N}\left( f \right)\left( \theta  \right) - \int_{{\mathbb{T}^\infty }} {f( {\hat \theta } )d\hat \theta } } \right\|_{\mathcal{B}}} \leqslant \frac{{{C_2 }}}{{{N^m}}},\;\;N \geqslant 1,
	\end{equation}
	and
	\begin{equation}\label{T2-2}
		{\left\| {\frac{1}{T}\int_0^T {w \left( t/T \right)f\left( {\rho t + \theta } \right)dt}  - \int_{{\mathbb{T}^\infty }} {f( {\hat \theta } )d\hat \theta } } \right\|_\mathcal{B}} \leqslant \frac{{{C_2 }}}{{{T^m}}},\;\;T \geqslant 1,
	\end{equation}
	where the positive constant $ C_2 $ only depends on $ f,{\tt{d}} ,{\tilde \Delta }_\infty,\eta ,w,\gamma ,m$.
\end{theorem}
\begin{remark}
	For the selected spatial structure and (\textbf{H2}), we eliminate the dependence of the universal constant on the dimension of the domain. This is extremely surprising because it avoids the Curse of Dimensionality.
\end{remark}

As an application of Theorem \ref{T3}, we give the following corollary based on the Infinite-dimensional Diophantine condition.

\begin{corollary}[Universality of arbitrary polynomial convergence via analyticity in the almost periodic case]\label{C1}
	Give $ f \in \mathcal{B}_{\tilde \Delta_\infty }  $. Assume that the irrational vector $ \rho $ satisfies the  Infinite-dimensional Diophantine condition in Definition  \ref{In} with  \eqref{diowuqiong}, and $ f $ is analytic in some neighbourhood of $ \mathbb{T}^\infty $ in $ \mathbb{C}^{\N} $.	 Then \eqref{333} and \eqref{T2-2} hold with arbitrary given $ 2 \leqslant m \in \N^+ $ and a positive constant $ C_3  $ that only depends on $ f,{\tilde \Delta }_\infty,\eta ,\mu ,w,\gamma ,m  $.
\end{corollary}
\begin{remark}\label{remacoro2.11}
	Analyticity  in the almost periodic case implies that the Fourier coefficients of $ f $ converge at an exponential rate under the  spatial structure (similar to (ii) in Comment (2)), i.e., the approximation function in \eqref{wuqiongkongjian} satisfies $ {{\tilde \Delta }_\infty }\left( x \right) = {\exp(x)} $ without loss of generality, see also \cite{MR4201442}. 
\end{remark}
\begin{remark}
Corollary \ref{C1} shows that arbitrary polynomial convergence is indeed universal via anaylticity in the almost periodic case, as long as the weighting function considered is $ \mathcal{C}_0^\infty\left( {\left[ {0,1} \right]} \right) $, since the Diophantine rotations form a set of full Lebesgue measure, see Remark \ref{wuqiongweimance}.
\end{remark}

\section{Convergence of exponential rate type} \label{sec3}
As mentioned in \eqref{renyi2} and \eqref{renyi}, the classical Birkhoff average might converge at an arbitrarily slow rate. Surprisingly, if we choose a weighting function good enough and require that the Fourier coefficients of $ f $ to converge more rapidly, then the corresponding weighted Birkhoff average could converge at an exponential rate. We also provide an intuitive explanation of why the exponential rate could be indeed achieved, see the proof of Theorem \ref{youxianweirenyisulv} in Subsection \ref{ProofofT3}.

Here we choose the new weighting function as 
\begin{equation}\label{barwde}
	\bar w\left( x \right): = {\left( {\int_0^1 {\exp \left( { - {s^{ - 1}}{{\left( {1 - s} \right)}^{ - 1}}} \right)ds} } \right)^{ - 1}} \cdot \exp \left( { - {x^{ - 1}}{{\left( {1 - x} \right)}^{ - 1}}} \right)
\end{equation}
on $ \left( {0,1} \right) $, i.e., $ p=q=1 $ in \eqref{zhishupingjun}, and let $ \bar{w}(0)=\bar{w}(1)=0 $. Denote by $ {\mathrm{\overline{WB}}_N}\left( f \right)\left( \theta  \right) $ the corresponding weighted Birkhoff average at this point. According to Lemma \ref{Lemma5.3}, we have the following $ L^1 $ norm estimates for the higher derivatives of $ \bar{w} $:
\begin{equation}\label{wn}
	\int_0^1 {\left| {{{\bar w}^{\left( n \right)}}\left( x \right)} \right|dx}  \leqslant  C_* { n ^{\beta n}},\;\;n \geqslant 2,
\end{equation}
provided with $ {C_ * } > 0 $ that only depends on $ \bar{w} $, and $ \beta >1 $ is an absolute constant. As we will see later, \eqref{wn} and the truncation technique will play an important role in dealing with exponential convergence for quasiperiodic and almost periodic cases. In fact, the resulting convergence rate will be faster if one can improve the upper bound in \eqref{wn} or find a better weighting function. However, we suspect that \textit{the hyperexponential convergence rate (e.g., $ \exp ( { - \exp({N})} ) $) cannot be achieved through this approach}, because higher derivatives in \eqref{wn} generally have coefficients such as $  n! \sim \sqrt {2\pi n} {\left( {n/e} \right)^n} $, and the former seems to  require that $ {\left\| {{{\bar w}^{\left( n \right)}}} \right\|_{{L^1}\left( {0,1} \right)}} = \mathcal{O}\left( {{{\left( {\log n} \right)}^n}} \right) $.

\begin{definition}[Adaptive function]
	A function $ \varphi(x) $ defined on $ \left[ {1, + \infty } \right) $ is called an adaptive function, if it is nondecreasing, satisfies that $ \varphi(+\infty)=+\infty $ and  $ \varphi \left( x \right) = o\left( x \right) $ as  $ x \to +\infty $.
\end{definition}
\begin{remark}
	For example, $ \varphi_1(x)=\log^{u}(1+x) $ with $ u>0 $ and $ {\varphi _2}\left( x \right) = {x^v} $ with $ 0<v<1 $ are all adaptive functions. The selection of an  adaptive function is important for the analysis of convergence rate below.
\end{remark}
We are now in a position to establish the exponential convergence theorems  through a given adaptive function $ \varphi $ and under certain assumptions.
\subsection{Finite-dimensional case $ \mathbb{T}^d $}
We make the following assumption:

(\textbf{H3})\ Let an adaptive function $ \varphi $ be given. The approximation functions given in \eqref{nonresonant}, \eqref{nonresonant2} and \eqref{Bdelta} satisfy the smallness condition with some $ c>0 $:
\[\int_{{\Delta ^{ - 1}}\left( {2\pi \alpha x/\varphi \left( x \right)} \right)}^{ + \infty } {\frac{{{r^{d - 1}}\Delta^2 \left( r \right)}}{{\tilde \Delta \left( r \right)}}dr}  = \mathcal{O}\left( {{e^{ - cx}}} \right),\;\;x \to  + \infty .\]

\begin{theorem}\label{youxianweirenyisulv}
	Give an adaptive function $ \varphi $, let $ f \in \mathcal{B}_{\tilde \Delta }  $,  and $ \rho $ satisfy the Finite-dimensional nonresonant condition in Definition \ref{Fi}. Assume (\textbf{H3}). Then there exist an absolute constant $ \beta_*>0 $, and a positive constant $ C_4 $ that only depends on $ {f,\alpha ,d,\Delta,\tilde{\Delta},\varphi,c} $ such that the following hold with $ N,T $ sufficiently large
	\begin{equation}\label{333-1}
		{\left\| { {\mathrm{\overline{WB}}_N}\left( f \right)\left( \theta  \right) - \int_{{\mathbb{T}^d}} {f( {\hat \theta } )d\hat \theta } } \right\|_\mathcal{B}} \leqslant {C_{4}}\exp\left(-(\varphi(N))^{\beta_*}\right),
	\end{equation}
	and
	\begin{equation}\label{333-2}
		{\left\| {\frac{1}{T}\int_0^T {\bar w\left( {t/T} \right)f\left( {\rho t + \theta } \right)dt}  - \int_{{\mathbb{T}^d}} {f( {\hat \theta } )d\hat \theta } } \right\|_\mathcal{B}} \leqslant {C_4}\exp\left(-(\varphi(T))^{\beta_*}\right).
	\end{equation}
\end{theorem}
\begin{remark}\label{remark3.1}
	It can be obviously seen from the L'Hospital's rule that if the divergence rate of $ \tilde{\Delta}(x) $ is rapid enough, then the convergence of the weighted Birkhoff average can indeed be of exponential rate type (because $ \Delta(x) $ is fixed at this point).  In fact, the smallness of (\textbf{H3}) can be further  weakened, we do not pursue that.
\end{remark}

Based on Theorem \ref{youxianweirenyisulv}, we give the following corollary to the case where $ f $ is analytic and $ \rho $ is Diophantine. The Gevrey smooth situation is in fact similar, which we omit here.

\begin{corollary}[Universality of exponential convergence via analyticity in the quasiperiodic case]\label{jiexi+dio}
	Assume that the irrational vector $ \rho $ satisfies the  Finite-dimensional Diophantine condition \eqref{Finite-dimensional Diophantine condition}, and $ f $ is analytic in some neighbourhood of $ \mathbb{T}^d $ in $ \mathbb{C}^d $. Then Theorem \ref{youxianweirenyisulv} holds with $ N,T $ sufficiently large and a universal constant $ C_5>0 $ independent of them, and the convergence rate is indeed exponential, i.e.,  $ \mathcal{O}(\exp ( { - \tilde c{N^{{\xi}}}} )) $ and $ \mathcal{O}(\exp ( { - \tilde c{T^{{\xi}}}} )) $ with some $ \tilde{c}>0 , \xi  = {\beta _ * }{\left( {1 + \tau {\beta _ * }} \right)^{ - 1}} > 0 $.
\end{corollary}
\begin{remark}
This corollary shows that exponential convergence is indeed universal in the quasiperiodic case via analyticity, since the Diophantine rotations form a set of full Lebesgue measure.
\end{remark}

\subsection{Infinite-dimensional case $ \mathbb{T}^\infty $}
We make the following assumption:

(\textbf{H4})\ Let an adaptive function $ \varphi $ be given. The approximation functions given in  \eqref{Infinite-dimensional nonresonant condition},  \eqref{Infinite-dimensional nonresonant condition2} and \eqref{wuqiongkongjian} satisfy the following smallness condition with some $ c>0 $:
\[\sum\limits_{{{\left| k \right|}_\eta } > {\tt{d} ^{ - 1}}\left( {2\pi \gamma x/\varphi \left( x \right)} \right)} {\frac{{{{\tt{d}}}^2 \left( {{{\left| k \right|}_\eta }} \right)}}{{{{\tilde \Delta }_\infty }\left( {{{\left| k \right|}_\eta }} \right)}}}  = \mathcal{O}\left( {{e^{ - cx}}} \right),\;\;x \to  + \infty .\]

\begin{theorem}\label{wuqiongweirenyisulv}
	Give an adaptive function $ \varphi $, let $ f \in \mathcal{B}_{\tilde \Delta_\infty }  $, and $ \rho $ satisfy the  Infinite-dimensional nonresonant condition in Definition \ref{In}.	Assume (\textbf{H4}). Then there exist an absolute constant $ \beta_*>0 $, and a positive constant $ C_6 $ that only depends on $ f,{\tt{d}} ,{\tilde \Delta }_\infty,\eta ,\gamma, \varphi,c $ such that the following hold with $ N,T $ sufficiently large
	\begin{equation}\notag
		{\left\| {{\mathrm{\overline{WB}}_N}\left( f \right)\left( \theta  \right) - \int_{{\mathbb{T}^\infty }} {f( {\hat \theta } )d\hat \theta } } \right\|_\mathcal{B}} \leqslant {C_6}\exp\left(-(\varphi(N))^{\beta_*}\right),
	\end{equation}
	and
	\begin{equation}\notag
		{\left\| {\frac{1}{T}\int_0^T {\bar w\left( {t/T} \right)f\left( {\rho t + \theta } \right)dt}  - \int_{{\mathbb{T}^\infty }} {f( {\hat \theta } )d\hat \theta } } \right\|_\mathcal{B}} \leqslant {C_6}\exp\left(-(\varphi(T))^{\beta_*}\right).
	\end{equation}
\end{theorem}

The convergence in Theorem \ref{wuqiongweirenyisulv} can be of indeed  exponential as long as the adaptive function $ \varphi(x)=\sqrt{x} $ is chosen and the divergence rate of $ \tilde{\Delta}_\infty(x) $ is rapid enough, similar to Remark \ref{remark3.1} and Corollary \ref{jiexi+dio}. We present the following corollary via Diophantine rotation  as an example, which is a special case   in Corollary \ref{C1} as we forego.

\begin{corollary}\label{wuqiongweidio}
	Give $ f \in \mathcal{B}_{\tilde \Delta_\infty }  $ with $ {{\tilde \Delta }_\infty }\left( x \right) = \exp \left( {\exp \left( x \right)} \right) $, and assume that $ \rho $ satisfies the  Infinite-dimensional Diophantine condition \eqref{diowuqiong} with $ 2\leqslant \mu = \eta \in \mathbb{N}^+ $.	 Then Theorem \ref{wuqiongweirenyisulv} holds with  exponential convergence rate, i.e., $ \mathcal{O}(\exp(-N^{\upsilon })) $ and $ \mathcal{O}(\exp(-T^{\upsilon })) $ with some $ \upsilon>0 $, as long as $ N,T $ are  sufficiently large.
\end{corollary}

\subsection{Cases without small divisors}\label{Cases without small divisors}
In fact, small divisors appear in the proof of Theorem \ref{T1} to Theorem \ref{wuqiongweirenyisulv} due to integration by parts, which not only brings difficulties to the proof, but also requires additional assumptions (such as (\textbf{H1}) to (\textbf{H4}), etc.), and even affects the convergence rate, e.g.,  if the divergence speed of $ \tilde{\Delta}(x) $ is so slow that the order of the integral in (\textbf{H3}) is only polynomial's type ($ N^{-m} $ with some $ m>0 $), then Theorem \ref{youxianweirenyisulv} might not admit exponential convergence. If we can  avoid the small divisors, then the above problems are solved and the resulting rate of convergence is certainly exponential. It should be pointed out that, for the discrete case with $ 1 \leqslant d \leqslant \infty $ and for the continuous case with $ 2 \leqslant d \leqslant \infty $, to avoid small divisors, one has to restrict $ f $ to  trigonometric polynomials, namely considering the following spaces
\[	{\mathcal{B}_{\tilde \Delta ,K}}: = \left\{ {f \in {\mathcal{B}_{\tilde \Delta }}:{\hat{f}_k} = 0 \text{ for all } \left\| k \right\| > K \in {\mathbb{N}^ + }} \right\},\]
and
\[	{\mathcal{B}_{\tilde \Delta_\infty ,K}}: = \left\{ {f \in {\mathcal{B}_{\tilde \Delta_\infty }}:{\hat{f}_k} = 0 \text{ for all } | k  |_\eta > K \in {\mathbb{N}^ + }} \right\},\]
provided a $ {K \in {\mathbb{N}^ + }} $. As to the continuous case with $ d=1 $, naturally there are no small divisors. The analysis becomes simpler than that before (in fact part of Theorem \ref{youxianweirenyisulv}) in the absence of  small divisors, we present Theorems \ref{without1} and \ref{without2} as follows.

\begin{theorem}\label{without1}
	Give $ f \in {\mathcal{B}_{\tilde \Delta ,K}} $ (or $f \in  \mathcal{B}_{\tilde \Delta_\infty ,K}  $). Then there exist some $ \hat{c}>0$  and ${C_7}>0 $ independent of $ N,T $, such that
	\[{\left\| {{\mathrm{\overline{WB}}_N}\left( f \right)\left( \theta  \right) - \int_{{\mathbb{T}^d}} {f( {\hat \theta } )d\hat \theta } } \right\|_\mathcal{B}} \leqslant {C_7}\exp \left( { - N^{\hat{c}}} \right),\;\;1\leqslant d \leqslant \infty\]
	and
	\[{\left\| {\frac{1}{T}\int_0^T {\bar w\left( {t/T} \right)f\left( {\rho t + \theta } \right)dt}  - \int_{{\mathbb{T}^d}} {f( {\hat \theta } )d\hat \theta } } \right\|_\mathcal{B}} \leqslant {C_7}\exp \left( { - T^{\hat{c}}} \right),\;\;2\leqslant d \leqslant \infty\]
	for $ N,T $ sufficiently large.
\end{theorem}

\begin{theorem}\label{without2}
	Give $ f \in \mathcal{B}_{\tilde{\Delta}} $ with $ d=1 $. Assume that
	\begin{equation}\label{without2-1}
		\sum\limits_{k \ne 0} {\frac{1}{{\tilde \Delta \left( {\left| k \right|} \right)}}}  <  + \infty .
	\end{equation}
	Then there exist some $ \hat{c}>0$  and ${C_8}>0 $ independent of $ T $, such that
	\[{\left\| {\frac{1}{T}\int_0^T {\bar w\left( {t/T} \right)f\left( {\rho t + \theta } \right)dt}  - \int_{{\mathbb{T}^1}} {f( {\hat \theta } )d\hat \theta } } \right\|_\mathcal{B}} \leqslant {C_8}\exp \left( { -  T^{\hat{c}}} \right)\]
	for $ T $ sufficiently large.
\end{theorem}

\section{Proof of results}
\subsection{Proof of Theorem \ref{T1}}
We first prove the discrete case \eqref{a.e.}, and some useful estimates should be provided.

Note that $ w \in \mathcal{C}_0^m\left( {\left( {0,1} \right)} \right) $, then there exists $ {C_w}>0 $ such that
\begin{equation}\label{Cw}
	\frac{N}{{{A_N}}} = {\left( {\frac{1}{N}\sum\limits_{n = 0}^{N - 1} {w\left( {\frac{n}{N}} \right)} } \right)^{ - 1}} \leqslant {C_w},\;\;\forall N \in \mathbb{N}^+.
\end{equation}

Integrating by parts $ m $ times yields that
\begin{align}
	\left| {\int_0^1 {w\left( y \right){e^{2N\pi i\left( {k \cdot \rho  - n} \right)y}}dy} } \right| &= \frac{1}{{2N\pi \left| {k \cdot \rho  - n} \right|}}\left| {\int_0^1 {w\left( y \right)d{e^{2N\pi i\left( {k \cdot \rho  - n} \right)y}}} } \right| \notag \\
	& = \frac{1}{{2N\pi \left| {k \cdot \rho  - n} \right|}}\left| {\int_0^1 {{w^{\left( 1 \right)}}\left( y \right){e^{2N\pi i\left( {k \cdot \rho  - n} \right)y}}dy} } \right|\notag \\
    & \cdots \notag \\
	&= \frac{1}{{{{\left( {2N\pi \left| {k \cdot \rho  - n} \right|} \right)}^m}}}\left| {\int_0^1 {{w^{\left( m \right)}}\left( y \right){e^{2N\pi i\left( {k \cdot \rho  - n} \right)y}}dy} } \right|\notag \\
	\label{Nm}& \leqslant \frac{1}{{{{\left( {2N\pi \left| {k \cdot \rho  - n} \right|} \right)}^m}}}||{w^{\left( m \right)}}||_{{{L^1}\left( {0,1} \right)}},
\end{align}
where we use the fact $ w \in \mathcal{C}_0^m\left( {\left( {0,1} \right)} \right) $ to eliminate the boundary terms.

For any fixed $ 0 \ne k \in {\mathbb{Z}^d} $, denote by $ {n_{k}} \in \mathbb{N} $  the closest integer to the number $ {k \cdot \rho } $. Note that $ m \geqslant 2 $. Therefore by \eqref{nonresonant} and $ \Delta(1)=1 $ we have
\begin{align}
	\sum\limits_{n = -\infty}^{+\infty} {\frac{1}{{{{\left| {k \cdot \rho  - n} \right|}^m}}}}  &= \frac{1}{{{{\left| {k \cdot \rho  - {n_{k}}} \right|}^m}}} + \sum\limits_{n \ne {n_{k}}} {\frac{1}{{{{\left| {k \cdot \rho  - n} \right|}^m}}}} \notag \\
	& \leqslant \frac{{{\Delta ^m}\left( {||k||} \right)}}{{{\alpha ^m}}} + 2\sum\limits_{n = 0}^\infty  {\frac{1}{{{{\left( {|n| + 1/2} \right)}^m}}}} \notag \\
	\label{deltam}&\leqslant {C_{\alpha ,m}}{\Delta ^m}\left( {||k||} \right),
\end{align}
because after being far away from the fixed number $ k \cdot \rho $, the rest summation of series naturally converges and is independent of the small divisor.

Note that
\[{\hat f_0} = \int_{{\mathbb{T}^d}} {f( \hat\theta  )d\hat\theta }  = {\mathrm{WB}_N}\left( {{f_0}} \right).\]
Then it follows that
\begin{align}
	{\varepsilon _N}\left(\theta\right): &={\left\| {{\mathrm{WB}_N}\left( f \right)\left( \theta  \right) - \int_{{\mathbb{T}^d}} {f( {\hat \theta } )d\hat \theta } } \right\|_{\mathcal{B}}}\notag \\
	& \leqslant \sum\limits_{0 \ne k \in {\mathbb{Z}^d}} {\left\| {{\hat  f_k}} \right\|_{\mathcal{B}}\left|{{\mathrm{WB}_N}\left( {{e^{2\pi ik \cdot \theta }}} \right)}\right|}\notag \\
	\label{xishu}& \leqslant {C_{f,\tilde{\Delta} }}\sum\limits_{0 \ne k \in {\mathbb{Z}^d}} {\frac{1}{{\tilde \Delta \left( {||k||} \right)}}\left| {\frac{1}{{{A_N}}}\sum\limits_{n = 0}^{N - 1} {w\left( {\frac{n}{N}} \right){e^{2\pi ik \cdot \left( {\theta  + n\rho } \right)}}} } \right|}  \\
	& = \frac{{{C_{f,\tilde{\Delta} }}}}{{{A_N}}}\sum\limits_{0 \ne k \in {\mathbb{Z}^d}} {\frac{1}{{\tilde \Delta \left( {||k||} \right)}}\left| {\sum\limits_{n = 0}^{N - 1} {w\left( {\frac{n}{N}} \right){e^{2\pi ink \cdot \rho }}} } \right|} \notag \\
\label{shuangwuqiong}	& = \frac{{{C_{f,\tilde{\Delta} }}}}{{{A_N}}}\sum\limits_{0 \ne k \in {\mathbb{Z}^d}} {\frac{1}{{\tilde \Delta \left( {||k||} \right)}}\left| {\sum\limits_{n = -\infty}^{+\infty} {w\left( {\frac{n}{N}} \right){e^{2\pi ink \cdot \rho }}} } \right|}  \\
	\label{poission}& = \frac{{{C_{f,\tilde{\Delta} }}}}{{{A_N}}}\sum\limits_{0 \ne k \in {\mathbb{Z}^d}} {\frac{1}{{\tilde \Delta \left( {||k||} \right)}}\left| {\sum\limits_{n = -\infty}^{+\infty} {\int_{ - \infty }^{ + \infty } {w\left( {\frac{t}{N}} \right){e^{2\pi itk \cdot \rho }}{e^{ - 2\pi itn}}dt} } } \right|}  \\
	\label{poissionpoission}& =C_{f,\tilde{\Delta} } \frac{{{}N}}{{{A_N}}}\sum\limits_{0 \ne k \in {\mathbb{Z}^d}} {\frac{1}{{\tilde \Delta \left( {||k||} \right)}}\sum\limits_{n = -\infty}^{+\infty} {\left| {\int_0^1 {w\left( y \right){e^{2N\pi i\left( {k \cdot \rho  - n} \right)y}}dy} } \right|} }   \\
	\label{2.4}&  \leqslant {C_{f,\tilde{\Delta} ,w}}\sum\limits_{0 \ne k \in {\mathbb{Z}^d}} {\frac{1}{{\tilde \Delta \left( {||k||} \right)}}\sum\limits_{n = -\infty}^{+\infty} {\left| {\int_0^1 {w\left( y \right){e^{2N\pi i\left( {k \cdot \rho  - n} \right)y}}dy} } \right|} }   \\
	\label{2.5}& \leqslant {C_{f,\tilde{\Delta} ,w}}\sum\limits_{0 \ne k \in {\mathbb{Z}^d}} {\frac{1}{{\tilde \Delta \left( {||k||} \right)}}\sum\limits_{n = -\infty}^{+\infty} {\frac{1}{{{{\left( {2N\pi \left| {k \cdot \rho  - n} \right|} \right)}^m}}}||{w^{\left( m \right)}}|{|_{{L^1}\left( {0,1} \right)}}} } \\
	&\leqslant \frac{{{C_{f,\tilde{\Delta} ,w,m}}}}{{{N^m}}}\sum\limits_{0 \ne k \in {\mathbb{Z}^d}} {\frac{1}{{\tilde \Delta \left( {||k||} \right)}}\sum\limits_{n = -\infty}^{+\infty} {\frac{1}{{{{\left| {k \cdot \rho  - n} \right|}^m}}}} } \notag \\
	\label{2.6}& \leqslant \frac{{{C_{f,\tilde{\Delta} ,w,m,\alpha }}}}{{{N^m}}}\sum\limits_{0 \ne k \in {\mathbb{Z}^d}} {\frac{{{\Delta ^m}\left( {||k||} \right)}}{{\tilde \Delta \left( {||k||} \right)}}} \\
	& \leqslant \frac{{{C_1 }}}{{{N^m}}}\int_1^{ + \infty } {\frac{{{r^{d - 1}}{\Delta ^m}\left( r \right)}}{{\tilde \Delta \left( r \right)}}dr} \notag \\
	\label{youjie}& \leqslant \frac{{{C_1 }}}{{{N^m}}},
\end{align}
provided a universal constant $ {C_1 } = {C_{f,\Delta ,\tilde \Delta ,w,\alpha ,m,d}} > 0 $. Here \eqref{Bdelta} is used in \eqref{xishu}, \eqref{shuangwuqiong} is because $ w(n/N)=0 $ for $ n \in \Z \backslash \{0,1, \cdots, N-1\} $ (note that $ w \in \mathcal{C}_0^m\left( {\left( {0,1} \right)} \right) $), the Poisson summation formula in Lemma \ref{po} is used in \eqref{poission},  \eqref{Cw} is used in \eqref{2.4}, \eqref{Nm} is used in \eqref{2.5}, \eqref{deltam} is used in \eqref{2.6}, and finally, \eqref{youjie} is because of (\textbf{H1}). This finishes the proof of the discrete case  \eqref{a.e.}.

As to the continuous case \eqref{T1-2}, the analysis is similar. This proves Theorem \ref{T1}.

\subsection{Proof of Theorem \ref{T3}}
We first prove the discrete case \eqref{333}. Note that
\begin{align*}
	\left| {\frac{1}{{{A_N}}}\sum\limits_{n = 0}^{N - 1} {w\left( {\frac{n}{N}} \right){e^{2\pi nik \cdot \rho }}} } \right| &\leqslant \frac{N}{{{A_N}}}\sum\limits_{n = -\infty}^{+\infty}\frac{1}{{{{\left( {2N\pi \left| {k \cdot \rho  - n} \right|} \right)}^m}}}{\left\| {{w^{\left( m \right)}}} \right\|_{{L^1}\left( {0,1} \right)}} \\
	&\leqslant \frac{{{C_{w,m}}}}{{{N^m}}}\sum\limits_{n = -\infty}^{+\infty}\frac{1}{{{{\left| {k \cdot \rho  - n} \right|}^m}}} \\
	&\leqslant \frac{{{C_{w,m,\gamma }}}}{{{N^m}}}{\tt{d}^m}\left( {{{\left| k \right|}_\eta }} \right).
\end{align*}
Then recalling the proof of Theorem \ref{T1} and (\textbf{H2}), we obtain that
\begin{align*}
	{\varepsilon _N}\left( \theta  \right): &= {\left\| {{\mathrm{WB}_N}\left( f \right)\left( \theta  \right) - \int_{{\mathbb{T}^\infty }} {f( {\hat \theta } )d\hat \theta } } \right\|_{\mathcal{B}}}\\
	& \leqslant \sum\limits_{0 \ne k \in \mathbb{Z}_ * ^\infty } {{{\left\| {{\hat f_k}} \right\|}_{\mathcal{B}}}{{\left| {{\mathrm{WB}_N}\left( {{e^{2\pi ik \cdot \theta }}} \right)} \right|}}} \\
	& = \sum\limits_{0 \ne k \in \mathbb{Z}_ * ^\infty } {{{\left\| {{\hat f_k}} \right\|}_{\mathcal{B}}}\left| {\frac{1}{{{A_N}}}\sum\limits_{n = 0}^{N - 1} {w\left( {\frac{n}{N}} \right){e^{2\pi ik \cdot \left( {\theta  + n\rho } \right)}}} } \right|} \\
	& \leqslant {C_{f,{{\tilde \Delta }_\infty }}}\sum\limits_{0 \ne k \in \mathbb{Z}_ * ^\infty } {\frac{1}{{{{\tilde \Delta }_\infty }\left( {{{\left| k \right|}_\eta }} \right)}}\left| {\frac{1}{{{A_N}}}\sum\limits_{n = 0}^{N - 1} {w\left( {\frac{n}{N}} \right){e^{2\pi nik \cdot \rho }}} } \right|} \\
	& \leqslant \frac{{{C_{f,{{\tilde \Delta }_\infty }}} \cdot {C_{w,m,\gamma }}}}{{{N^m}}}\sum\limits_{0 \ne k \in \mathbb{Z}_ * ^\infty } {\frac{{{{\tt d}^m}\left( {{{\left| k \right|}_\eta }} \right)}}{{{{\tilde \Delta }_\infty }\left( {{{\left| k \right|}_\eta }} \right)}}} \\
	& \leqslant \frac{{{C_2 }}}{{{N^m}}},
\end{align*}
where the positive constant $ C_2 $ only depends on $ f,{\tt{d}} ,{\tilde \Delta }_\infty,\eta ,w,\gamma ,m $. We therefore finish the proof of  the discrete case \eqref{333}.

As to the continuous case  \eqref{T2-2}, the analysis is similar. This proves Theorem \ref{T3}.

\subsection{Proof of Corollary \ref{C1}}
One just needs to verify (\textbf{H2}).

Note that for all $\mu  > 1 $ and fixed $ \rho_*  = 1/2m>0$, we have
\[{\tt d}\left( {{{\left| k \right|}_\eta }} \right) = \prod\limits_{j \in \mathbb{N}} {\left( {1 + {{\left| {{k_j}} \right|}^\mu }{{\left\langle j \right\rangle }^\mu }} \right)}  \leqslant \exp \left( {\frac{\tau }{{{\rho_* ^{1/\eta }}}}\log \left( {\frac{\tau }{\rho_* }} \right)} \right) \cdot {e^{\rho_* {{\left| k \right|}_\eta }}}\]
with some $ \tau  = \tau \left( {\eta ,\mu } \right) > 0 $ for all $ 0 \ne k \in \mathbb{Z}_ * ^\infty  $, see Lemma \ref{5.2}.  Recall $ 2 \leqslant \eta  \in {\mathbb{N}^ + } $, then
\begin{equation}\label{fan}
	{\left| k \right|_\eta } = \sum\limits_{j \in \mathbb{N}} {{{\left\langle j \right\rangle }^\eta }\left| {{k_j}} \right|}  \in {\mathbb{N}^ + }.
\end{equation}
Thus
\begin{align}
	\sum\limits_{0 \ne k \in \mathbb{Z}_ * ^\infty } {\frac{{{{\tt d}^m}\left( {{{\left| k \right|}_\eta }} \right)}}{{{{\tilde \Delta }_\infty }\left( {{{\left| k \right|}_\eta }} \right)}}}  &\leqslant {C_{\eta ,\mu,m  }}\sum\limits_{0 \ne k \in \mathbb{Z}_ * ^\infty } {\frac{{{e^{m\rho_* {{\left| k \right|}_\eta }}}}}{{{{\tilde \Delta }_\infty }\left( {{{\left| k \right|}_\eta }} \right)}}} \notag \\
	& = {C_{\eta ,\mu,m  }}\sum\limits_{\nu  = 1}^\infty  {\left( {\sum\limits_{0 \ne k \in \mathbb{Z}_ * ^\infty ,{{\left| k \right|}_\eta } = \nu } {\frac{1}{{{{\tilde \Delta }_\infty }\left( {{{\left| k \right|}_\eta }} \right)}}{e^{m\rho_* {{\left| k \right|}_\eta }}}} } \right)} \notag \\
	\label{3-1}& = {C_{\eta ,\mu,m  }}\sum\limits_{\nu  = 1}^\infty  {\left( {\frac{1}{{{{\tilde \Delta }_\infty }\left( \nu  \right)}}{e^{m\rho_* \nu }}\sum\limits_{0 \ne k \in \mathbb{Z}_ * ^\infty ,{{\left| k \right|}_\eta } = \nu } 1 } \right)} .
\end{align}
Denote
\begin{equation}\label{hhhhh}
	\sum\limits_{0 \ne k \in \mathbb{Z}_ * ^\infty ,{{\left| k \right|}_\eta } = \nu } 1  = \# \left\{ {k:0 \ne k \in \mathbb{Z}_ * ^\infty ,{{\left| k \right|}_\eta } = \nu  \in {\mathbb{N}^ + }} \right\}.
\end{equation}
Hence, in view of \eqref{fan}, the  largest non-zero integer $ {j_{\max }} $ in \eqref{hhhhh}  satisfies $ {j_{\max }} \leqslant \left[ {{\nu ^{1/\eta }}} \right] $, and that's why we need a certain spatial structure. Therefore, we have
\begin{align}
	&\# \left\{ {k:0 \ne k \in \mathbb{Z}_ * ^\infty ,{{\left| k \right|}_\eta } = \nu  \in {\mathbb{N}^ + }} \right\} \notag \\
	\leqslant &\# \left\{ {k:0 \ne k \in \mathbb{Z}_ * ^\infty ,|{k_0}| + |{k_1}| +  \cdots  + |{k_{\left[ {{\nu ^{1/\eta }}} \right]}}| = \nu  \in {\mathbb{N}^ + }} \right\}\notag \\
	\leqslant &{2^{\left[ {{\nu ^{1/\eta }}} \right] + 1}} \cdot \# \left\{ {k:0 \ne k \in \mathbb{Z}_ * ^\infty ,\text{$ {k_j} \in \mathbb{N} $ for all $ j \in \mathbb{N} $},{k_0} + {k_1} +  \cdots  + {k_{\left[ {{\nu ^{1/\eta }}} \right]}} = \nu  \in {\mathbb{N}^ + }} \right\}\notag \\
	= &{2^{\left[ {{\nu ^{1/\eta }}} \right] + 1}} \cdot C_{\nu  + \left[ {{\nu ^{1/\eta }}} \right]}^\nu \notag \\
	\label{CCCC}\leqslant &{2^{\left[ {{\nu ^{1/\eta }}} \right] + 1}} \cdot {C_\eta }\frac{1}{{\sqrt {{\nu ^{1/\eta }}} }} \cdot {\nu ^{\left( {1 - 1/\eta } \right)\left( {{\nu ^{1/\eta }} + 1} \right)}} \cdot {e^{\left[ {{\nu ^{1/\eta }}} \right]}} \\
	\label{sssss} \leqslant &  {C_\eta }{\nu ^{{\nu ^{1/\eta }}}}.
\end{align}
Here \eqref{CCCC} uses the following fact:
\begin{align*}
	&\;\;\;\;\;\;C_{\nu  + \left[ {{\nu ^{1/\eta }}} \right]}^\nu \\ 
	&= \frac{{\left( {\nu  + \left[ {{\nu ^{1/\eta }}} \right]} \right)!}}{{\nu !\left( {\left[ {{\nu ^{1/\eta }}} \right]} \right)!}} \sim \frac{{\sqrt {2\pi \left( {\nu  + \left[ {{\nu ^{1/\eta }}} \right]} \right)} {{\left( {\frac{{\nu  + \left[ {{\nu ^{1/\eta }}} \right]}}{e}} \right)}^{\nu  + \left[ {{\nu ^{1/\eta }}} \right]}}}}{{\sqrt {2\pi \nu } {{\left( {\frac{\nu }{e}} \right)}^\nu } \cdot \sqrt {2\pi \left[ {{\nu ^{1/\eta }}} \right]} {{\left( {\frac{{\left[ {{\nu ^{1/\eta }}} \right]}}{e}} \right)}^{\left[ {{\nu ^{1/\eta }}} \right]}}}}\\
	& \sim \frac{1}{{\sqrt {2\pi {\nu ^{1/\eta }}} }} \cdot {\left( {1 + \frac{{\left[ {{\nu ^{1/\eta }}} \right]}}{\nu }} \right)^\nu } \cdot {\left( {\frac{\nu }{{\left[ {{\nu ^{1/\eta }}} \right]}}} \right)^{\left[ {{\nu ^{1/\eta }}} \right]}} \cdot {\left( {1 + \frac{{\left[ {{\nu ^{1/\eta }}} \right]}}{\nu }} \right)^{\left[ {{\nu ^{1/\eta }}} \right]}}\\
	& = \frac{1}{{\sqrt {2\pi {\nu ^{1/\eta }}} }} \cdot {\left( {\frac{\nu }{{\left[ {{\nu ^{1/\eta }}} \right]}}} \right)^{\left[ {{\nu ^{1/\eta }}} \right]}} \cdot \exp \left( {\nu \log \left( {1 + \frac{{\left[ {{\nu ^{1/\eta }}} \right]}}{\nu }} \right)} \right) \\
	&\;\;\;\;\cdot \exp \left( {\left[ {{\nu ^{1/\eta }}} \right]\log \left( {1 + \frac{{\left[ {{\nu ^{1/\eta }}} \right]}}{\nu }} \right)} \right)\\
	& = \frac{1}{{\sqrt {2\pi {\nu ^{1/\eta }}} }} \cdot {\left( {\frac{\nu }{{\left[ {{\nu ^{1/\eta }}} \right]}}} \right)^{\left[ {{\nu ^{1/\eta }}} \right]}} \cdot \exp \left( {\nu \left( {\frac{{\left[ {{\nu ^{1/\eta }}} \right]}}{\nu } - \frac{1}{2}\frac{{{{\left[ {{\nu ^{1/\eta }}} \right]}^2}}}{{{\nu ^2}}} +  \cdots } \right)} \right) \\
	&\;\;\;\;\cdot \exp \left( {\left[ {{\nu ^{1/\eta }}} \right]\left( {\frac{{\left[ {{\nu ^{1/\eta }}} \right]}}{\nu } +  \cdots } \right)} \right)\\
	& = \frac{1}{{\sqrt {2\pi {\nu ^{1/\eta }}} }} \cdot {\left( {\frac{\nu }{{\left[ {{\nu ^{1/\eta }}} \right]}}} \right)^{\left[ {{\nu ^{1/\eta }}} \right]}} \cdot \exp \left( {\left[ {{\nu ^{1/\eta }}} \right] - \frac{{{{\left[ {{\nu ^{1/\eta }}} \right]}^2}}}{{2\nu }} +  \cdots } \right) \cdot \exp \left( {\frac{{{{\left[ {{\nu ^{1/\eta }}} \right]}^2}}}{\nu } +  \cdots } \right)\\
	& = \frac{1}{{\sqrt {2\pi {\nu ^{1/\eta }}} }} \cdot {\left( {\frac{\nu }{{\left[ {{\nu ^{1/\eta }}} \right]}}} \right)^{\left[ {{\nu ^{1/\eta }}} \right]}} \cdot \exp \left( {\left[ {{\nu ^{1/\eta }}} \right] + \mathcal{O}\left( 1 \right)} \right) \cdot \exp \left( {\mathcal{O}\left( 1 \right)} \right) (\text{since $ \eta \geqslant 2 $})\\
	& \leqslant {C_\eta }\frac{1}{{\sqrt {{\nu ^{1/\eta }}} }} \cdot {\nu ^{\left( {1 - 1/\eta } \right)\left( {{\nu ^{1/\eta }} + 1} \right)}} \cdot {e^{\left[ {{\nu ^{1/\eta }}} \right]}}.
\end{align*}

Finally, combining \eqref{3-1}, \eqref{hhhhh} and \eqref{sssss} we arrive at
\begin{align}
	\sum\limits_{0 \ne k \in \mathbb{Z}_ * ^\infty } {\frac{{{{\tt d}^m}\left( {{{\left| k \right|}_\eta }} \right)}}{{{{\tilde \Delta }_\infty }\left( {{{\left| k \right|}_\eta }} \right)}}} & \leqslant {C_{\eta ,\mu,m  }}\sum\limits_{\nu  = 1}^\infty  {\left( {\frac{1}{{{{\tilde \Delta }_\infty }\left( \nu  \right)}}{e^{m\rho_* \nu }}\sum\limits_{0 \ne k \in \mathbb{Z}_ * ^\infty ,{{\left| k \right|}_\eta } = \nu } 1 } \right)} \notag \\
	&\leqslant {C_{\eta ,\mu ,m }}\sum\limits_{\nu  = 1}^\infty  {\left( {\frac{1}{{{{\tilde \Delta }_\infty }\left( \nu  \right)}}{e^{m\rho_* \nu }} \cdot {C_{\eta  } }{\nu ^{{\nu ^{1/\eta }}}}} \right)} \notag \\
	&  \leqslant {C_{\eta ,\mu ,m}}\sum\limits_{\nu  = 1}^\infty  {\left( {\frac{{{\nu ^{{\nu ^{1/\eta }}}}}}{{{e^{\nu \left( {1 - m\rho_* } \right)}}}}} \right)}  \notag \\
	& = {C_{\eta ,\mu ,m}}\sum\limits_{\nu  = 1}^\infty  {\left( {\frac{1}{{{e^{\nu /2 - {\nu ^{1/\eta }}\log \nu }}}}} \right)} \notag \\
	& \leqslant {C_{\eta ,\mu ,m}}\sum\limits_{\nu  = 1}^\infty  {\left( {\frac{1}{{{e^{\nu /4}}}}} \right)} \notag \\
	& \leqslant C_3 \notag\\
	&<  + \infty
\end{align}
for a universal positive constant $ C_ 3 $, since $ {{\tilde \Delta }_\infty }\left( x \right) = {\exp(x)} $ (see Remark \ref{remacoro2.11}), i.e., (\textbf{H2}) holds. Then we finish the proof by applying Theorem \ref{T3}.

\subsection{Proof of Theorem \ref{youxianweirenyisulv}}\label{ProofofT3}
Intuitively, let's first present an explanation for why the exponential rate can be achieved. For cases without small divisors, let us take $ f $ be a trigonometric polynomial as an example. In view of the estimates \eqref{wn} of the higher order derivatives of the weighting function $ \bar{w} $, we could change the times of integration by parts to achieve the fastest convergence rate under this approach (monotonicity analysis is sufficient), that is, an exponential convergence. Therefore for a general $ f $, if its Fourier coefficients converge rapidly enough, then intuitively it behaves like a trigonometric polynomial. One only needs to truncate the Fourier series into the principal and remainder terms with respect to the given $ N \in \mathbb{N}^+ $ and the chosen adaptive function $ \varphi(x) $ at this point. Specifically, for the principal term we could perform the above operation (integration by parts of varying times), and for the remainder we just employ the analysis of Theorem \ref{T1} (integration by parts of fixed times).

We first prove the discrete case \eqref{333-1}, and the proof is divided into four steps.
\\
{\bf{Step1:}} For given  adaptive function $  \varphi  $  and $ N \in \mathbb{N}^+$ sufficiently large, define
\begin{align*}
	{\Lambda _1}&: = \left\{ {k:0 \ne k \in {\mathbb{Z}^d},||k|| \leqslant {\Delta ^{ - 1}}\left( {2\pi \alpha N/\varphi \left( N \right)} \right)} \right\},\\
	{\Lambda _2}&: = \left\{ {k:0 \ne k \in {\mathbb{Z}^d},||k|| > {\Delta ^{ - 1}}\left( {2\pi \alpha N/\varphi \left( N \right)} \right)} \right\}.
\end{align*}
This gives $ {\Lambda _1}\bigcap {{\Lambda _2}} {\text{ = }}\phi  $ and $ {\Lambda _1}\bigcup {{\Lambda _2}} {\text{ = }}\left\{ {k:0 \ne k \in {\mathbb{Z}^d}} \right\} $. Further, one notices that $ \left| {{\Lambda _1}} \right|,\left| {{\Lambda _2}}\right| \to  + \infty  $ when $ N \to +\infty $ because $ \varphi(x)=o(x) $ and $ \Delta^{-1}(+\infty)=+\infty $. At this point, we have
\begin{align}
	&{\left\| {{\mathrm{\overline{WB}}_N}\left( f \right)\left( \theta  \right) - \int_{{\mathbb{T}^d}} {f( {\hat \theta } )d\hat \theta } } \right\|_\mathcal{B}}\notag \\
	\leqslant &{C_{f,\tilde \Delta }}\sum\limits_{0 \ne k \in {\mathbb{Z}^d}} {\frac{1}{{\tilde \Delta \left( {\left\| k \right\|} \right)}}\sum\limits_{n = -\infty}^{+\infty} {\left| {\int_0^1 {\bar w\left( y \right){e^{2\pi iN\left( {k \cdot \rho  - n} \right)y}}dy} } \right|} } \notag \\
	\leqslant &{C_{f,\tilde \Delta }}\left( \begin{gathered}
		\sum\limits_{k \in {\Lambda _1}} {\frac{1}{{\tilde \Delta \left( {||k||} \right)}}\sum\limits_{n =  - \infty }^{ + \infty } {\left| {\int_0^1 {\bar w\left( y \right){e^{2\pi iN\left( {k \cdot \rho  - n} \right)y}}dy} } \right|} }  \hfill \\
		+ \sum\limits_{k \in {\Lambda _2}} {\frac{1}{{\tilde \Delta \left( {||k||} \right)}}\sum\limits_{n =  - \infty }^{ + \infty } {\left| {\int_0^1 {\bar w\left( y \right){e^{2\pi iN\left( {k \cdot \rho  - n} \right)y}}dy} } \right|} }  \hfill \\
	\end{gathered}  \right)\\
	\label{P3--1}: = &{C_{f,\tilde \Delta }}\left( {{\tt{S}_1} + {\tt{S}_2}} \right)
\end{align}
according to \eqref{poissionpoission} in the proof of Theorem \ref{T1}, where $ \tt{S}_1 $ and $ \tt{S}_2 $ represent the principal term and the remainder term, respectively.
\\
{\bf{Step2:}} For the principal term $ \tt{S}_1 $, we choose
\[{L_1} = {L_1}\left( {k,N} \right): = \left[ {{e^{ - 1}}{{\left( {\frac{{\Delta \left( {||k||} \right)}}{{2\pi \alpha N}}} \right)}^{ - 1/\beta }}} \right] \geqslant 2\]
for fixed $ k \in \Lambda_1 $ and $ N \in {\mathbb{N}^ + } $ sufficiently large, where $ \beta>0 $ is the absolute constant given in \eqref{wn}. One can verify that $ \mathop {\inf }\limits_{k \in {\Lambda _1},N \in {\mathbb{N}^ + }} {L_1} =  + \infty  $, which implies that the times of integration by parts become infinite when $ N \to +\infty$. Further, it follows that
\begin{equation}\label{mNguji}
	{\left( {\frac{{L_1^\beta \Delta \left( {||k||} \right)}}{{2\pi \alpha N}}} \right)^{{L_1}}} \leqslant C_{\alpha,\varphi,\Delta}\exp \left( { - {{\left( {\varphi \left( N \right)} \right)}^{\beta^*}}} \right)
\end{equation}
with $ \beta^*=(2\beta)^{-1}>0$ (also an absolute constant) for all $ k\in \Lambda_1 $. Note that $ L_1 \geqslant 2 $ as long as $ N $ sufficiently large. We therefore derive that
\begin{align}
	\label{S_1-1}{\tt{S}_1} &\leqslant \sum\limits_{k \in {\Lambda _1}} {\frac{1}{{\tilde \Delta \left( {||k||} \right)}}\left( \begin{gathered}
			{\left\| {{{\bar w}^{\left( {{L_1}} \right)}}} \right\|_{{L^1}\left( {0,1} \right)}}{\left( {\frac{{\Delta \left( {||k||} \right)}}{{2\pi \alpha N}}} \right)^{{L_1}}} \hfill \\
			+ 2{\left\| {{{\bar w}^{\left( {{L_1}} \right)}}} \right\|_{{L^1}\left( {0,1} \right)}}\sum\limits_{n = 0}^{ + \infty } {\frac{1}{{{{\left( {2\pi N\left( {n + 1/2} \right)} \right)}^{{L_1}}}}}}  \hfill \\
		\end{gathered}  \right)} \\
	\label{S_1-2}& \leqslant \sum\limits_{k \in {\Lambda _1}} {\frac{1}{{\tilde \Delta \left( {||k||} \right)}}\left( \begin{gathered}
			{\left\| {{{\bar w}^{\left( {{L_1}} \right)}}} \right\|_{{L^1}\left( {0,1} \right)}}{\left( {\frac{{\Delta \left( {||k||} \right)}}{{2\pi \alpha N}}} \right)^{{L_1}}} \hfill \\
			+ 2{\left\| {{{\bar w}^{\left( {{L_1}} \right)}}} \right\|_{{L^1}\left( {0,1} \right)}}{\left( {\frac{1}{{2\pi N}}} \right)^{{L_1}}}\sum\limits_{n = 0}^{ + \infty } {\frac{1}{{{{\left( {n + 1/2} \right)}^{{L_1}}}}}}  \hfill \\
		\end{gathered}  \right)} \\
	\label{S_1-3}&\leqslant {C_{\alpha,\varphi,\Delta}}\sum\limits_{k \in {\Lambda _1}} {\frac{1}{{\tilde \Delta \left( {||k||} \right)}}{{\left( {\frac{{L_1^\beta \Delta \left( {||k||} \right)}}{{2\pi \alpha N}}} \right)}^{{L_1}}}} \\
	\label{S_1-4}& \leqslant {{C_{\alpha,\varphi,\Delta} }}\sum\limits_{k \in {\Lambda _1}} {\frac{1}{{\tilde \Delta \left( {||k||} \right)}} \cdot \exp \left( { - {{\left( {\varphi \left( N \right)} \right)}^{\beta^*}}} \right)} \\
	& \leqslant {C_{\alpha,\varphi,\Delta}}\left( {\sum\limits_{0 \ne k \in {\mathbb{Z}^d}} {\frac{1}{{\tilde \Delta \left( {||k||} \right)}}} } \right) \cdot \exp \left( { - {{\left( {\varphi \left( N \right)} \right)}^{\beta^* }}} \right)\notag\\
	& \leqslant {C_{\alpha,\varphi,\Delta,\tilde \Delta,d}}\left( {\int_1^{ + \infty } {\frac{{{r^{d - 1}}}}{{\tilde \Delta \left( r \right)}}dr} } \right) \cdot \exp \left( { - {{\left( {\varphi \left( N \right)} \right)}^{\beta^* }}} \right)\notag\\
	\label{S_1-5}& \leqslant {C_{\alpha,\varphi,\Delta,\tilde \Delta,d}}\exp \left( { - {{\left( {\varphi \left( N \right)} \right)}^{\beta^* }}} \right).
\end{align}
Here \eqref{S_1-1} is same as \eqref{Nm} and \eqref{deltam} in the proof of Theorem \ref{T1}. \eqref{S_1-3} uses \eqref{wn}, and it shows that the second term in parentheses in \eqref{S_1-2} is relatively small compared to the first one. \eqref{S_1-4} uses \eqref{mNguji}, and finally \eqref{S_1-5} is because
\[\int_1^{ + \infty } {\frac{{{r^{d - 1}}}}{{\tilde \Delta \left( r \right)}}dr}  = \mathcal{O}\left( {\int_1^{ + \infty } {\frac{{{r^{d - 1}}\Delta^2 \left( r \right)}}{{\tilde \Delta \left( r \right)}}dr} } \right) = \mathcal{O}\left( 1 \right)\]
due to (\textbf{H3}) and  Cauchy's Theorem.
\\
{\bf{Step3:}} As to the  remainder term $ \tt{S}_2 $, similar to \eqref{Nm} and \eqref{deltam} with $ m=2 $ we arrive at
\begin{align}
	{\tt{S}_2} &\leqslant \sum\limits_{||k|| > {\Delta ^{ - 1}}\left( {2\pi \alpha N/\varphi \left( N \right)} \right)} {\frac{1}{{\tilde \Delta \left( {||k||} \right)}}\sum\limits_{n =  - \infty }^{ + \infty } {\left| {\int_0^1 {\bar w\left( y \right){e^{2\pi iN\left( {k \cdot \rho  - n} \right)y}}dy} } \right|} } \notag \\
	\label{jibiji}& \leqslant \frac{{{C_1}}}{N^2}\sum\limits_{||k|| > {\Delta ^{ - 1}}\left( {2\pi \alpha N/\varphi \left( N \right)} \right)} {\frac{{\Delta^2 \left( {||k||} \right)}}{{\tilde \Delta \left( {||k||} \right)}}} \\
	& \leqslant \frac{{{C_1} \cdot {C_{\alpha ,\Delta ,\tilde \Delta,\varphi }}}}{N^2}\int_{{\Delta ^{ - 1}}\left( {2\pi \alpha N/\varphi \left( N \right)} \right)}^{ + \infty } {\frac{{{r^{d - 1}}\Delta^2 \left( r \right)}}{{\tilde \Delta \left( r \right)}}dr}\notag \\
	\label{S_2-3}& \leqslant {C_1} \cdot {C_{\alpha ,\Delta ,\tilde \Delta,\varphi }}\int_{{\Delta ^{ - 1}}\left( {2\pi \alpha N/\varphi \left( N \right)} \right)}^{ + \infty } {\frac{{{r^{d - 1}}\Delta^2 \left( r \right)}}{{\tilde \Delta \left( r \right)}}dr}  \\
	\label{S_2-1}& \leqslant {C_1} \cdot {C_{\alpha ,\Delta ,\tilde \Delta,\varphi }}\exp(-cN),
\end{align}
where $ C_1>0 $ is the universal constant in Theorem \ref{T1}, and (\textbf{H3}) is used in \eqref{S_2-1}.
\\
{\bf{Step4:}} By substituting \eqref{S_1-5} and \eqref{S_2-1} into \eqref{P3--1} we immediately have
\[{\left\| {{\mathrm{\overline{WB}}_N}\left( f \right)\left( \theta  \right) - \int_{{\mathbb{T}^d}} {f( {\hat \theta } )d\hat \theta } } \right\|_\mathcal{B}} \leqslant {C_4}\exp \left( { - {{\left( {\varphi \left( N \right)} \right)}^{\beta^* }}} \right)\]
for $ N $ sufficiently large, provided a positive constant $ C_4>0 $ that only depends on $ {f,\alpha ,d,\Delta,\tilde{\Delta},\varphi},c $. Then we finish the proof of the discrete case \eqref{333-1}.

The analysis is similar for the continuous case \eqref{333-2}, and this proves Theorem \ref{youxianweirenyisulv}.

\subsection{Proof of Corollary \ref{jiexi+dio}}
At this point, we have $ \Delta \left( x \right) = {x^\tau } $ with $ \tau >d-1 $ and $ {\Delta ^{ - 1}}\left( x \right) = {x^{1/\tau }} $. In view the analyticity of $ f $, there exist $ c_f,\mu  > 0 $  such that $ | {{{\hat f}_k}} | \leqslant c_f{e^{ - 2\mu ||k||}} $ for all $ 0 \ne k \in {\mathbb{Z}^d} $, i.e., $ \tilde \Delta \left( x \right) = {e^{2\mu x}} $. However, (\textbf{H3}) does not hold. Recall Remark \ref{remark3.1}, we could choose an  appropriate adaptive function $ \varphi $ and slightly modify the proof of Theorem \ref{youxianweirenyisulv} (i.e., the analysis of $ \tt{S}_1 $ and $ \tt{S}_2 $) to obtain exponential convergence. Let $ \varphi \left( x \right) = {x^\varpi } $ with $ \varpi  = {\left( {1 + \tau {\beta _ * }} \right)^{ - 1}} \in \left( {0,1} \right) $, where $ \beta _ *>0 $ is the constant given in Theorem \ref{youxianweirenyisulv}. Then it follows that
\[\exp \left( { - {{\left( {\varphi \left( x \right)} \right)}^{{\beta ^*}}}} \right) = \exp \left( { - {x^{\varpi {\beta ^*}}}} \right) = \exp \left( { - {x^\xi }} \right)\]
in \eqref{S_1-5}, where $ \xi  = \beta ^*{\left( {1 + \tau {\beta _ * }} \right)^{ - 1}}\in (0,1) $, and
\begin{align*}
	\int_{{\Delta ^{ - 1}}\left( {2\pi \alpha x/\varphi \left( x \right)} \right)}^{ + \infty } {\frac{{{r^{d - 1}}\Delta^2 \left( r \right)}}{{\tilde \Delta \left( r \right)}}dr}  &= \int_{{{\left( {2\pi \alpha {x^{1 - \varpi }}} \right)}^{1/\tau }}}^{ + \infty } {\frac{{{r^{d + 2\tau  - 1}}}}{{{e^{2\mu r}}}}dr} \\
	& = \mathcal{O}\left( {\int_{{{\left( {2\pi \alpha {x^{1 - \varpi }}} \right)}^{1/\tau }}}^{ + \infty } {{e^{ - \mu r}}dr} } \right)\\
	& = \mathcal{O}\left( {\exp \left( { - \tilde c{x^{\left( {1 - \varpi } \right){\tau ^{ - 1}}}}} \right)} \right)\\
	& = \mathcal{O}\left( {\exp \left( { - \tilde c{x^\xi }} \right)} \right)
\end{align*}
in \eqref{S_2-3} with some $ \tilde{c}>0 $. Therefore Corollary \ref{jiexi+dio} is proved by \eqref{S_1-5}, \eqref{S_2-1} and  \eqref{P3--1}.

\subsection{Proof of Theorem \ref{wuqiongweirenyisulv}}
Recall \eqref{jibiji}, then the proof is the same as Theorem \ref{youxianweirenyisulv} due to (\textbf{H4}).

\subsection{Proof of Corollary \ref{wuqiongweidio}}
One just needs to verify (\textbf{H4}). We omit a few calculations here for brevity. Recall $ {{\tilde \Delta }_\infty }\left( x \right) = \exp \left( {\exp \left( x \right)} \right) $ and \eqref{diowuqiong}. Note Lemma \ref{5.2} with $ 2\leqslant \mu = \eta \in \mathbb{N}^+ $ implies that $ {\tt{d}}(x) = \mathcal{O}(e^x) $, and thus $ \log x = \mathcal{O}\left( {{{\tt{d}}^{ - 1}}\left( x \right)} \right) $. Let $ \varphi(x)=\sqrt{x} $, we therefore derive that
\begin{align}
	\sum\limits_{{{\left| k \right|}_\eta } \geqslant {{\tt{d}}^{ - 1}}\left( 2\pi \gamma x/\varphi \left( x \right) \right)} {\frac{{\tt{d}}^2\left({\left| k \right|}_\eta\right)}{{{{\tilde \Delta }_\infty }\left( {{{\left| k \right|}_\eta }} \right)}}}  &= \sum\limits_{\nu  \geqslant {{\tt{d}}^{ - 1}}\left(  2\pi \gamma x/\varphi \left( x \right)  \right)} {\frac{{\tt{d}}^2(v)}{{{{\tilde \Delta }_\infty }\left( \nu  \right)}}\left( {\sum\limits_{0 \ne k \in \mathbb{Z}_ * ^\infty ,{{\left| k \right|}_\eta } = \nu } 1 } \right)} \notag \\
	&=\mathcal{O}\left( {\sum\limits_{v \geqslant \tilde c\log x} {\frac{{{e^{2v}}}}{{\exp \left( {{e^v}} \right)}} \cdot {v^{{\nu ^{1/\eta }}}}} } \right)\notag \\
	& = \mathcal{O}\left( {\sum\limits_{v \geqslant \tilde c\log x} {\frac{1}{{\exp \left( {{e^v}/2} \right)}}} } \right)\notag \\
	& = \mathcal{O}\left( {{e^{ - \tilde cx}}} \right)\notag,
\end{align}
provided a universal constant $ \tilde c>0 $,  and \eqref{hhhhh}, \eqref{sssss} are used here. One notices that the convergence rate at this point is of exponential's type, i.e., $ \mathcal{O}(\exp(-N^{\upsilon })) $ and $ \mathcal{O}(\exp(-T^{\upsilon })) $ with some $ \upsilon>0 $ according to Theorem \ref{wuqiongweirenyisulv}, we therefore finish the proof.

\subsection{Proof of Theorems \ref{without1} and \ref{without2}}
Consider Theorem \ref{without1}. Note that there are no small divisors at this point, we therefore could slightly modify the proof of Theorem \ref{youxianweirenyisulv}.  Specifically,  $ \Lambda_1 $ is a finite set and $ \Lambda_2=\phi $ as long as $ N $ is sufficiently large,  we thus only need to estimate the principal term $ \tt{S}_1 $. The convergence rate is indeed exponential through  the same technique of integration by parts, see \eqref{mNguji}. This gives the proof.

As to Theorem \ref{without2}, the proof is similar since the  estimates obtained by integration by parts are exponentially small for all $ 0 \ne k \in \mathbb{Z} $ as long as $ T $ is sufficiently large, and the universal coefficient will be guaranteed boundedness by \eqref{without2-1}, we therefore finish the proof.

\section{Appendix}
\begin{lemma}[Poisson summation formula] \label{po}
	For each $ h\left( x \right) \in {L^2}\left( \mathbb{R} \right) $,  there holds
	\begin{equation}\notag
		\sum\limits_{n \in \mathbb{Z}} {h\left( n \right)}  = \sum\limits_{n \in \mathbb{Z}} {\int_{ - \infty }^{ + \infty } {h\left( x \right){e^{ - 2\pi nix}}dx} } .
	\end{equation}
\end{lemma}
\begin{proof}
	See Chapter 3 in \cite{MR2445437} for  details.
\end{proof}

\begin{lemma}\label{5.2}
	For arbitrary given $ \rho_*  > 0$ and $\mu  \in {\mathbb{N}^ + } $, there exists $ \tau  = \tau \left( {\eta ,\mu } \right) > 0 $ such that
	\[ \prod\limits_{j \in \mathbb{N}} {\left( {1 + {{\left| {{k_j}} \right|}^\mu }{{\left\langle j \right\rangle }^\mu }} \right)}  \leqslant \exp \left( {\frac{\tau }{{{\rho_* ^{1/\eta }}}}\log \left( {\frac{\tau }{\rho_* }} \right)} \right) \cdot {e^{\rho_* {{\left| k \right|}_\eta }}}.\]
\end{lemma}
\begin{proof}
	See details in Lemma B.1 in \cite{MR4201442} and  Lemma 7.2 in \cite{MR4091501}.
\end{proof}

\begin{lemma}\label{Lemma5.3}
	Define 
	\[\bar w\left( x \right): = {\left( {\int_0^1 {\exp \left( { - {s^{ - 1}}{{\left( {1 - s} \right)}^{ - 1}}} \right)ds} } \right)^{ - 1}} \cdot \exp \left( { - {x^{ - 1}}{{\left( {1 - x} \right)}^{ - 1}}} \right)\]
	 on $ \left( {0,1} \right) $. Then the following holds with  $ C_*={\left( {\int_0^1 {\exp \left( { - {s^{ - 1}}{{\left( {1 - s} \right)}^{ - 1}}} \right)ds} } \right)^{ - 1}}>0 $, where  $  \beta>1  $ is some  universal absolute constant:
	\begin{equation}\label{L1}
		\int_0^1 {\left| {{{\bar w}^{\left( n \right)}}\left( x \right)} \right|dx}  \leqslant C_*{ n ^{\beta n}},\;\;n \geqslant 2.
	\end{equation}
\end{lemma}
\begin{proof}
	We're going to prove \eqref{L1} in four steps.
	\\
	{\bf{Step1:}} Define $ P\left( x \right): = {e^{ - \frac{1}{x}}} $. Then it follows that $ \bar w\left( x \right) = C_*P\left( x \right)P\left( {1 - x} \right)$ and $ \left({P\left( x \right)}\right)^{\left( 1 \right)} = \frac{1}{{{x^2}}}{e^{ - \frac{1}{x}}} $.
	One can verify the following by induction:
	\begin{equation}\label{l1}
		\left({P\left( x \right)}\right)^{\left( n \right)} = \left( {\frac{1}{{{x^{2n}}}} + \frac{{a_{2n - 1}^{\left( n \right)}}}{{{x^{2n - 1}}}} +  \cdots  + \frac{{a_1^{\left( n \right)}}}{{{x^1}}}} \right){e^{ - \frac{1}{x}}},\;\;n \geqslant 1,
	\end{equation}
	where $ a_j^{\left( n \right)} \in \mathbb{Z} $ for all $ 1 \leqslant j \leqslant 2n - 1 $, and we define $ a_{2n}^{\left( n \right)}: = 1 $. At this point, denote $ {b_n}: = \mathop {\max }\limits_{1 \leqslant j \leqslant 2n} \left| {a_j^{\left( n \right)}} \right| \in {\mathbb{N}^ + } $, then $ b_1=1 $. In view of \eqref{l1}, we get
	\[{b_{n + 1}} \leqslant 2n \cdot 4n \cdot \mathop {\max }\limits_{1 \leqslant j \leqslant 2n} \left| {a_j^{\left( n \right)}} \right| = 8{n^2}{b_n},\]
	since when taking the derivative of \eqref{l1},  there are $ 4n $ terms that haven't been combined yet. Therefore, we have
	\begin{equation}\label{bbbbb}
		{b_n} \leqslant \prod\limits_{j = 1}^n {{{\left( {8j} \right)}^2}}  \cdot {b_1} = {8^n}{\left( {n!} \right)^2},\;\;n \geqslant 1.
	\end{equation}
	\\
	{\bf{Step2:}} Note that
	\[\mathop {\sup }\limits_{1/2 < s < 1} \left| {P{{\left( s \right)}^{\left( 0 \right)}}} \right| = \mathop {\sup }\limits_{1/2 < s < 1} {e^{ - \frac{1}{s}}} = {e^{ - 1}}.\]
	For all $ \;n \geqslant 1 $, by using \eqref{bbbbb} we get
	\begin{align}
		\mathop {\sup }\limits_{1/2 < s < 1} \left| {{{\left(P(s)\right)}^{\left( n \right)}}} \right| &= \mathop {\sup }\limits_{1/2 < s < 1} \left| {\left( {\frac{1}{{{s^{2n}}}} + \frac{{a_{2n - 1}^{\left( n \right)}}}{{{s^{2n - 1}}}} +  \cdots  + \frac{{a_1^{\left( n \right)}}}{{{s^1}}}} \right){e^{ - \frac{1}{s}}}} \right|\notag \\
		& \leqslant \mathop {\sup }\limits_{1/2 < s < 1} 2n \cdot \frac{{{b_n}}}{{{s^{2n}}}}{e^{ - \frac{1}{s}}}\notag \\
		& \leqslant 2n \cdot {4^n} \cdot {8^n}{\left( {n!} \right)^2}\notag \\
		&\leqslant {2^{6n }}{\left( {n!} \right)^2}.
	\end{align}
	Then we arrive at
	\begin{equation}\label{Stpe2}
		\mathop {\sup }\limits_{1/2 < s < 1} \left| {{{\left(P(s)\right)}^{\left( n \right)}}} \right| \leqslant {2^{6n }}{\left( {n!} \right)^2}, \;\; n \geqslant 0.
	\end{equation}
	\\
	{\bf{Step3:}} For $ n=0 $, we have
	\begin{equation}\label{ndengyu0}
		\int_0^{\frac{1}{2}} {\left| {{{\left( {P\left( x \right)} \right)}^{\left( 0 \right)}}} \right|dx}  = \int_0^{\frac{1}{2}} {{e^{ - \frac{1}{x}}}dx}  = \int_2^{ + \infty } {\frac{1}{{{y^2}}}{e^{ - y}}dy}  \leqslant 4\int_2^{ + \infty } {{e^{ - y}}dy}  = \frac{4}{{{e^2}}}.
	\end{equation}
	As to $ n \geqslant 1 $, by using \eqref{bbbbb} we have
	\begin{align}
		\int_0^{\frac{1}{2}} {\left| {{{\left( {P\left( x \right)} \right)}^{\left( n \right)}}} \right|dx}  &= \int_0^{\frac{1}{2}} {\left| {\left( {\frac{1}{{{x^{2n}}}} + \frac{{a_{2n - 1}^{\left( n \right)}}}{{{x^{2n - 1}}}} +  \cdots  + \frac{{a_1^{\left( n \right)}}}{{{x^1}}}} \right){e^{ - \frac{1}{x}}}} \right|dx} \notag \\
		& \leqslant \int_0^{\frac{1}{2}} {\frac{{2n \cdot {b_n}}}{{{x^{2n}}}}{e^{ - \frac{1}{x}}}dx} \notag \\
		& = n{2^{3n+1}}{\left( {n!} \right)^2}\int_2^{ + \infty } {{y^{2n - 2}}{e^{ - y}}dy} \notag \\
		& \leqslant n{2^{3n+1}}{\left( {n!} \right)^2}\int_0^{ + \infty } {{y^{2n - 2}}{e^{ - y}}dy} \notag \\
		& = n{2^{3n+1}}{\left( {n!} \right)^2} \cdot \left( {2n - 2} \right)!\notag \\
		\label{ndayu1}& \leqslant {2^{3n+1}}{\left( {n!} \right)^2}\left( {2n} \right)!.
	\end{align}
	\\
	{\bf{Step4:}} In view of \eqref{Stpe2}, \eqref{ndengyu0} \eqref{ndayu1} and the Stirling formula $ n! \sim \sqrt {2\pi n} {\left( {n/e} \right)^n} $, we finally arrive at
	\begin{align*}
		\int_0^1 {\left| {{{\bar w}^{\left( n \right)}}\left( x \right)} \right|dx} &=2\int_0^{\frac{1}{2}} {\left| {{{\bar w}^{\left( n \right)}}\left( x \right)} \right|dx}\\
		 &= 2C_*\int_0^{\frac{1}{2}} {\left| {\sum\limits_{i = 0}^n {C_n^i{{\left( {P\left( x \right)} \right)}^{\left( i \right)}}{{\left( {P\left( {1 - x} \right)} \right)}^{\left( {n - i} \right)}}} } \right|dx} \\
		& \leqslant 2C_*\sum\limits_{i = 0}^n {C_n^i} \int_0^{\frac{1}{2}} {\left| {{{\left( {P\left( x \right)} \right)}^{\left( i \right)}}} \right| \cdot \left| {{{\left( {P\left( {1 - x} \right)} \right)}^{\left( {n - i} \right)}}} \right|dx} \\
		& \leqslant 2C_*\sum\limits_{i = 0}^n {C_n^i} \int_0^{\frac{1}{2}} {\left| {{{\left( {P\left( x \right)} \right)}^{\left( i \right)}}} \right| \cdot \mathop {\sup }\limits_{1/2 < s < 1} \left| {\left(P{{\left( s \right)}}\right)^{\left( {n - i} \right)}} \right|dx} \\
		& \leqslant C_*{2^{6n+1}}{\left( {n!} \right)^2}\sum\limits_{i = 0}^n {C_n^i} \int_0^{\frac{1}{2}} {\left| {{{\left( {P\left( x \right)} \right)}^{\left( i \right)}}} \right|dx} \\
		& \leqslant C_*{2^{6n+1}}{\left( {n!} \right)^2}\left( {\frac{4}{{{e^2}}} + \sum\limits_{i = 1}^n {C_n^i} \int_0^{\frac{1}{2}} {\left| {{{\left( {P\left( x \right)} \right)}^{\left( i \right)}}} \right|dx} } \right)\\
		& \leqslant C_*{2^{6n+1}}{\left( {n!} \right)^2}\left( {\left( {\sum\limits_{i = 0}^n {C_n^i} } \right) \cdot {2^{3n+1}}{{\left( {n!} \right)}^2}\left( {2n} \right)!} \right)\\
		& = C_*{2^{10n+2}}{\left( {n!} \right)^4}\left( {2n} \right)!\\
		& \leqslant C_*{2^{10n+2}}{\left( {\frac{{{n^n}}}{{{e^n}}}} \right)^4}\left( {\frac{{{{\left( {2n} \right)}^{2n}}}}{{{e^{2n}}}}} \right)\\
		& \leqslant C_*{2^{10n+2}}\left( {\frac{{{n^{4n}}}}{{{2^{4n}}}}} \right)\left( {\frac{{{2^{2n}}{n^{2n}}}}{{{2^{2n}}}}} \right)\\
		& \leqslant C_*{2^{6n + 2}}{n^{6n}}\\
		& \leqslant C_* {n^{\beta n}}
		.
	\end{align*}
	This proves \eqref{L1} as long as if we choose $ \beta>6 $ independent of $ \bar w $ sufficiently large.

\end{proof}

\section*{Acknowledgments}
This work was supported by National Basic Research Program of China (grant No. 2013CB834100), National Natural Science Foundation of China (grant No. 11571065, 11171132, 12071175), Project of Science and Technology Development of Jilin Province, China (grant No. 2017C028-1, 20190201302JC), and Natural Science Foundation of Jilin Province (grant No. 20200201253JC).


\begin{thebibliography}{99}
	\bibitem{MR4091501}
	L. Biasco, J. E. Massetti, M. Procesi, 
	An abstract {B}irkhoff normal form theorem and exponential
	type stability of the 1d {NLS}.
	Comm. Math. Phys. 375 (2020), no. 3, 2089-2153. \href{https://mathscinet.ams.org/mathscinet/search/publdoc.html?arg3=&co4=AND&co5=AND&co6=AND&co7=AND&dr=all&pg4=AUCN&pg5=TI&pg6=PC&pg7=ALLF&pg8=ET&r=1&review_format=html&s4=&s5=An%20abstract%20Birkhoff%20normal%20form%20theorem%20and%20exponential%20type%20stability%20of%20the%201D%20NLS&s6=&s7=&s8=All&sort=Newest&vfpref=html&yearRangeFirst=&yearRangeSecond=&yrop=eq}{MR4091501}
	
	\bibitem{MR2180074}
	J. Bourgain,
	On invariant tori of full dimension for 1{D} periodic {NLS}.
	J. Funct. Anal. 229 (2005), no. 1, 62-94. \href{https://mathscinet.ams.org/mathscinet/search/publdoc.html?arg3=&co4=AND&co5=AND&co6=AND&co7=AND&dr=all&pg4=AUCN&pg5=TI&pg6=PC&pg7=ALLF&pg8=ET&r=1&review_format=html&s4=&s5=On%20invariant%20tori%20of%20full%20dimension%20for%201D%20periodic%20NLS&s6=&s7=&s8=All&sort=Newest&vfpref=html&yearRangeFirst=&yearRangeSecond=&yrop=eq}{MR2180074}
	
	\bibitem{MR1963683}
	M. Brin,  G. Stuck,
	Introduction to dynamical systems. Cambridge University Press, Cambridge, 2002. xii+240 pp. ISBN: 0-521-80841-3 \href{https://mathscinet.ams.org/mathscinet/search/publdoc.html?r=1&pg1=MR&s1=1963683&loc=fromrevtext}{MR1963683}
	
	
	
	
	
	\bibitem{MR3718733}
	S. Das, Y. Saiki, E. Sander, J. A. Yorke,
	Quantitative quasiperiodicity. Nonlinearity 30 (2017), no. 11, 4111-4140. \href{https://mathscinet.ams.org/mathscinet/search/publdoc.html?arg3=&co4=AND&co5=AND&co6=AND&co7=AND&dr=all&pg4=AUCN&pg5=TI&pg6=MR&pg7=ALLF&pg8=ET&r=1&review_format=html&s4=&s5=&s6=MR3718733&s7=&s8=All&sort=Newest&vfpref=html&yearRangeFirst=&yearRangeSecond=&yrop=eq}{MR3718733}	
	
	\bibitem{MR3755876}
	S. Das, J. A. Yorke, 
	Super convergence of ergodic averages for quasiperiodic orbits. Nonlinearity 31 (2018), no. 2, 491-501. \href{https://mathscinet.ams.org/mathscinet/search/publdoc.html?arg3=&co4=AND&co5=AND&co6=AND&co7=AND&dr=all&pg4=AUCN&pg5=TI&pg6=PC&pg7=ALLF&pg8=ET&r=1&review_format=html&s4=&s5=Super%20convergence%20of%20ergodic%20averages%20%20for%20quasiperiodic%20orbits&s6=&s7=&s8=All&sort=Newest&vfpref=html&yearRangeFirst=&yearRangeSecond=&yrop=eq}{MR3755876}
	
	
	\bibitem{MR2445437}
	L. Grafakos, 
	Classical Fourier analysis.
	Second edition. Graduate Texts in Mathematics, 249. Springer, New York, 2008. xvi+489 pp. ISBN: 978-0-387-09431-1 \href{https://mathscinet.ams.org/mathscinet/search/publdoc.html?arg3=&batch_title=Selected%20Matches%20for%3A%20Title%3D%28Classical%20%20Fourier%20%20Analysis%29&co4=AND&co5=AND&co6=AND&co7=AND&dr=all&fmt=doc&pg4=AUCN&pg5=TI&pg6=PC&pg7=ALLF&pg8=ET&review_format=html&s4=&s5=Classical%20%20Fourier%20%20Analysis&s6=&s7=&s8=All&searchin=&sort=citations&vfpref=html&yearRangeFirst=&yearRangeSecond=&yrop=eq&r=1&mx-pid=2445437}{MR2445437}
	
	\bibitem{Herman}
	M. R. Herman, 
	Une m\'{e}thode pour minorer les exposants de {L}yapounov et
	quelques exemples montrant le caract\`ere local d'un th\'{e}or\`eme
	d'{A}rnold et de {M}oser sur le tore de dimension {$2$}.
	Comment. Math. Helv. 58 (1983), no. 3, 453-502. \href{https://mathscinet.ams.org/mathscinet/search/publdoc.html?agg_year_1983=1983&batch_title=Selected%20Matches%20for%3A%20Items%20authored%20by%20Herman%2C%20Michael-Robert&fmt=doc&pg1=INDI&s1=84755&searchin=&sort=newest&vfpref=html&r=5&mx-pid=727713}{MR0727713}
	
	\bibitem{Moser}
	R. Johnson, J. Moser, 
	The rotation number for almost periodic potentials.
	Comm. Math. Phys. 84 (1982), no. 3, 403-438. \href{https://mathscinet.ams.org/mathscinet/search/publdoc.html?agg_year_1982=1982&batch_title=Selected%20Matches%20for%3A%20Items%20authored%20by%20Moser%2C%20J%C3%BCrgen%20K.&fmt=doc&pg1=INDI&r=1&s1=127360&searchin=&sort=newest&vfpref=html}{MR0667409}
	
	
	
	\bibitem{MR0510630}
	U. Krengel, 
	On the speed of convergence in the ergodic theorem.
	Monatsh. Math. 86 (1978/79), no. 1, 3-6. \href{https://mathscinet.ams.org/mathscinet/search/publdoc.html?arg3=&co4=AND&co5=AND&co6=AND&co7=AND&dr=all&pg4=AUCN&pg5=TI&pg6=MR&pg7=ALLF&pg8=ET&r=1&review_format=html&s4=&s5=&s6=MR0510630&s7=&s8=All&sort=Newest&vfpref=html&yearRangeFirst=&yearRangeSecond=&yrop=eq}{MR0510630}
	

	
	
	\bibitem{MR1720890}
	J. Laskar, 
	Introduction to frequency map analysis.  Hamiltonian systems with three or more degrees of freedom ({S}'{A}gar\'{o}, 1995), 134–150,
	NATO Adv. Sci. Inst. Ser. C: Math. Phys. Sci., 533, Kluwer Acad. Publ., Dordrecht, 1999.
	\href{https://mathscinet.ams.org/mathscinet/search/publdoc.html?batch_title=Selected%20Matches%20for%3A%20Items%20authored%20by%20Laskar%2C%20Jacques&fmt=doc&pg1=INDI&r=1&s1=110475&searchin=Introduction%20to%20Frequency%20Map%20Analysis&sort=newest&vfpref=html}{MR1720890}
	
	\bibitem{Mackey}
	G. W. Mackey,
	Ergodic theory and its significance for statistical mechanics and probability theory.
	Advances in Math. 12 (1974), 178-268. \href{https://mathscinet.ams.org/mathscinet/search/publdoc.html?arg3=&co4=AND&co5=AND&co6=AND&co7=AND&dr=all&pg4=AUCN&pg5=TI&pg6=PC&pg7=ALLF&pg8=ET&r=1&review_format=html&s4=&s5=%20Ergodic%20theory%20and%20its%20significance%20for%20statistical%20mechanics%20and%20probability%20theory&s6=&s7=&s8=All&sort=Newest&vfpref=html&yearRangeFirst=&yearRangeSecond=&yrop=eq}{MR0346131}
	
	
	
	\bibitem{MR4201442}
	R. Montalto, M. Procesi, 
	Linear {S}chr\"{o}dinger equation with an almost periodic potential.
	SIAM J. Math. Anal. 53 (2021), no. 1, 386-434. \href{https://mathscinet.ams.org/mathscinet/search/publdoc.html?agg_journal_SIAM%20J.%20Math.%20Anal.=SIAM%20J.%20Math.%20Anal.&batch_title=Selected%20Matches%20for%3A%20Citations%20of%202180074&fmt=doc&r=1&refcit=2180074&searchin=&sort=newest&vfpref=html}{MR4201442}
	
	\bibitem{ccm}
	C. C. Moore, 
	Ergodic theorem, ergodic theory, and statistical mechanics. 	Proc. Natl. Acad. Sci. USA 112 (2015), no. 7, 1907-1911.
	\href{https://mathscinet.ams.org/mathscinet/search/publdoc.html?arg3=&co4=AND&co5=AND&co6=AND&co7=AND&dr=all&pg4=AUCN&pg5=TI&pg6=PC&pg7=ALLF&pg8=ET&r=1&review_format=html&s4=&s5=Ergodic%20theorem%2C%20ergodic%20theory%2C%20and%20statistical%20mechanics&s6=&s7=&s8=All&sort=Newest&vfpref=html&yearRangeFirst=&yearRangeSecond=&yrop=eq}{MR3324732}
\end{thebibliography}
\end{document}